\documentclass{article}

\usepackage[english]{babel}

\usepackage[letterpaper,top=2cm,bottom=2cm,left=3cm,right=3cm,marginparwidth=1.75cm]{geometry}

\usepackage{amsmath}
\usepackage{amsthm}
\usepackage{graphicx}
\usepackage[colorlinks=true, allcolors=blue]{hyperref}

\usepackage{caption}
\usepackage{subcaption}

\usepackage{booktabs}
\graphicspath{{fig/}} 
\usepackage{algorithm}
\usepackage[noend]{algpseudocode}
\makeatletter
\def\BState{\State\hskip-\ALG@thistlm}
\makeatother

\usepackage{amssymb}
\usepackage{float} 
\usepackage{framed}

\newcommand*{\mb}[1]{{\mathbf{#1}}}

\newtheorem{theorem}{Theorem}[section]
\newtheorem{lemma}[theorem]{Lemma}

\newtheorem{example}[theorem]{Example}

\DeclareMathSymbol{\shortminus}{\mathbin}{AMSa}{"39}
\newtheorem{remark}[theorem]{Remark}

\DeclareMathSymbol{\shortminus}{\mathbin}{AMSa}{"39}
\DeclareMathOperator*{\RANK}{rank}

\title{Efficient least squares approximation and collocation \\methods using radial basis functions}
\author{Yiqing Zhou and Daan Huybrechs}

\begin{document}
\maketitle

\begin{abstract}
We describe an efficient method for the approximation of functions using radial basis functions (RBFs), and extend this to a solver for boundary value problems on irregular domains. The method is based on RBFs with centers on a regular grid defined on a bounding box, with some of the centers outside the computational domain. The equation is discretized using collocation with oversampling, with collocation points inside the domain only, resulting in a rectangular linear system to be solved in a least squares sense. The goal of this paper is the efficient solution of that rectangular system. We show that the least squares problem splits into a regular part, which can be expedited with the FFT, and a low rank perturbation, which is treated separately with a direct solver. The rank of the perturbation is influenced by the irregular shape of the domain and by the weak enforcement of boundary conditions at points along the boundary. The solver extends the AZ algorithm which was previously proposed for function approximation involving frames and other overcomplete sets. The solver has near optimal log-linear complexity for univariate problems, and loses optimality for higher-dimensional problems but remains faster than a direct solver. A Matlab implementation is available at \url{https://gitlab.kuleuven.be/numa/public/2023-RBF-AZalgorithm}.
\end{abstract}

\section{Introduction}
\label{ss:intro}

Several families of methods for the solution of PDEs on irregular domains revolve around embedding the domain inside a bounding box. Examples include the immersed boundary method~\cite{peskin_2002}, volume penalty methods~\cite{babuvska1973finite,hester2021improving} and Fourier embedding methods~\cite{elghaoui1999mixed, elghaoui1996spectral}, among others. The motivation of these methods is to reuse existing efficient solvers on regular domains, along with regular discretizations, in the solution of the irregular problem.

For the simpler problem of function approximation on irregular domains, that aim is achieved by the recently proposed AZ algorithm~\cite{coppe2020az}. The AZ algorithm originated in the solution of Fourier extension problems, in which a function on an irregular domain is extended to a periodic function on a larger bounding box~\cite{lyon2011fast,matthysen2016fastfe,matthysen2017fastfe2d}. The algorithm exploits the fact that the approximation problem on the irregular domain is a low rank perturbation of the  approximation problem on the larger domain. In the case of Fourier extensions, the latter is efficiently solved by the FFT. The AZ algorithm couples this FFT solver with a direct solver for the low-rank perturbation

Extension approaches naturally lead to ill-conditioned linear systems. The underlying reason is that the extension of a function to a larger domain is not unique. Hence, there might be multiple different solutions of the larger problem. Yet, this particular cause of ill-conditioning does not preclude high accuracy numerical approximations, since any of those multiple solutions correspond to the same (unique) solution of the problem on the smaller domain. The stability of computations with these and other overcomplete sets is analyzed using the theoretical concept of frames in~\cite{adcock2019frames,adcock2020frames}. A practical outcome is that numerical stability can be achieved under certain conditions, such as a frame condition~\cite{christensen2016introduction}, and in particular by using rectangular systems rather than square systems. A highly accurate solution can be found numerically if one exists for which the coefficients in the representation, i.e., the solution vector of the linear system, has small or moderate coefficient norm.

In the context of differential equations, rectangular systems may arise from oversampled collocation methods. The analysis based on frames was extended to boundary integral equations in~\cite{maierhofer2022oversampled,maierhofer2022collocationbie}. When using radial basis functions (RBFs), a collocation method for PDE's is often referred to as Kansa's method~\cite{kansa1990method2}. {Least squares collocation methods with RBFs have been proposed at least in~\cite{platte2006,PiretFrames,tominec2021rbfFD,frykland2018,tominec2021rbffd2,tominec2022rbf, hashemi2021,adcock2022rbf_frames,larsson2017rbfpartition}. }The discretization method and the computational results in this paper are similar to those in the given references. The focus of the paper is a novel -- and more efficient -- way to solve the associated ill-conditioned rectangular linear system, starting from the AZ algorithm for extension frames. For that reason, in contrast to the setting of some of the references -- in particular those focusing on partition of unity and finite-difference approaches -- we aim for a global approximant rather than a local one.

Several fast methods for the computation of RBF approximations have been proposed, see, e.g., ~\cite{biancolini2017fast, wendland2006computational} and references therein. The listed methods include among others fast multi-pole method~\cite{cherrie2000fast}, a domain decomposition method~\cite{beatson2001fast, wendland2004scattered} and multilevel methods~\cite{ohtake20053d}. Other methods based on parallelization have also been proposed, such as~\cite{bollig2012solution}.

\subsection{Main results}
We describe an efficient solver for the problem of approximating smooth functions on irregular domains, as well as the solution of boundary value problems. We consider 1D and 2D problems. In all cases, the irregular domain is embedded into a larger bounding box, on which a periodic problem is defined.

\begin{itemize}
    \item We show that the periodic problem on the larger domain can be solved efficiently using the FFT in all cases, even when the linear system is rectangular.
    \item For 1D function approximation we employ Gaussian RBF's with equispaced centers on a larger interval. A variant of the AZ algorithm is formulated to solve the non-periodic problem as a low-rank perturbation of the larger periodic problem, with near optimal computational complexity.
    \item The algorithm is essentially the same for 2D problems, though somewhat more involved. The computational complexity improves upon cubic complexity of a direct solver for the linear system, but remains quadratic rather than log-linear.
    \item We show that the solution of boundary value problems requires only a small modification of the proposed AZ algorithm, namely to incorporate boundary conditions. This adds to the rank of the low-rank perturbation but, somewhat surprisingly, the computational complexity remains the same as that of the simpler problem of function approximation.
\end{itemize}

The solver also achieves high accuracy and stability in all cases, in spite of the ill-conditioning of the linear system.

\subsection{Outline of the paper}
The structure of the paper is as follows. In \S\ref{ss:1d_fa} we introduce the univariate periodized Gaussian radial basis functions for the approximation of univariate periodic and non-periodic functions. In \S\ref{ss:eff_alg_1d} we describe the proposed efficient solver for the associated linear systems, along with numerical results. We analyze the accuracy, stability and efficiency of the algorithm in \S\ref{ss:theo}.

In \S\ref{ss:2d_fa} we introduce periodic and non-periodic two-dimensional function approximation problems and corresponding efficient solvers. Here, we focus on the differences with the one-dimensional problems. Finally, in \S\ref{ss:bvp} we modify the function approximation solver to solve elliptic boundary value problems in 1D and 2D. Conclusions and some directions for future research are given in \S
\ref{ss:con}.

\section{Univariate function approximation}
\label{ss:1d_fa}

\subsection{Radial basis functions and their periodization}

The approximation to a function using RBFs can be written as 
\begin{equation}
f(\mathbf{x}) \approx \sum_{j=1}^N a_j \phi ( \lVert	\mathbf{x} - \mathbf{c}_j \rVert   ), \quad \mathbf{x} \in  \mathbb{R} ^d.
\end{equation}
where $\phi(r)$ is a  univariate function, called the \emph{radial basis function}, and $\Vert \cdot \Vert$ denotes the Euclidean distance between the point $\mathbf{x} $ and the center $\mathbf{c}_j$. The coefficients $a_j$ typically result from solving a linear system. Smooth RBF's often lead to spectral convergence when approximating smooth functions~\cite{buhmann2003radial}.  

In this paper we will mostly use the standard Gaussian RBF, defined with a \emph{shape parameter} $\varepsilon$ as
\begin{equation}\label{eq:rbf}
    \phi(r) =  \exp ( -\varepsilon^2 r^2).
\end{equation}
Moreover, we adopt the setting of \emph{periodized} RBF approximations of~\cite{PiretFrames,adcock2022rbf_frames}. The radial basis function is periodized with period $2T$ by summing its translates,
\begin{equation}\label{eq:periodic_rbf}
\phi^{\rm per}(x) \equiv \sum_{m=-\infty}^{\infty} \phi ( x - 2mT ).
\end{equation}
The sum in the right hand side converges whenever $\phi$ has compact support or, in the case of the Gaussian RBF, converges owing to the rapid decay of $\phi$. In the following we will frequently omit the superscript and implicitly refer to the periodized basis functions.

The choice of the shape parameter strongly influences the convergence behaviour of RBF approximations. {The study in~\cite{adcock2022rbf_frames} indicates that the linear regime $\varepsilon = cN$ yields optimal (algebraic) convergence rates $\mathcal{O}(N^{-k})$ for functions in the $H^k(\Omega)$ Sobolev space. This regime results in ill-conditioned linear systems, which are regularized with a threshold $\tau_0$ that can be chosen to be very small. A suitable proportionality constant $c$ is explicitly identified which ensures that accuracy up to the level of $\tau_0$ can be reached:
}
\begin{equation}\label{eq:constant_shapeparam}
    c = \frac{\pi}{ 2 T \sqrt{2 \log ( 1+ \tau_0 ^{-2})}},
\end{equation}
This choice of shape parameter guarantees the existence of accurate RBF approximations with moderate coefficient norm $\Vert \mb{a} \Vert$. {In turn, that is why} in spite of having to solve ill-conditioned linear systems such approximations are numerically computable with standard least squares methods~\cite{adcock2019frames,adcock2020frames,adcock2022rbf_frames}. We refer to~\cite[\S4.1 and \S4.2]{adcock2022rbf_frames} for a precise statement on the convergence behaviour in this setting.

\begin{remark}
We focus on the Gaussian RBF for several reasons. It is a popular choice in a range of applications and a stability analysis of least squares approximations is available in~\cite{adcock2022rbf_frames}. The algorithm could also be implemented for other smooth kernels with some localization, such that periodization is applicable, and we do so in \S\ref{ss:bvp}. In comparison to  spline kernels (investigated in~\cite{coppe2022splines}), (i) smooth kernels typically have faster than algebraic convergence rates and (ii) the current algorithm does not exploit sparsity in the linear system. In comparison to the original AZ algorithm of~\cite{coppe2020az}, the RBF setting also raises new stability issues, which we fully address for the Gaussian kernel in \S\ref{ss:theo}. Finally, we note that our 2D experiments are specific to the Gaussian kernel as we exploit the fact that the product of two Gaussians is again a Gaussian function.
\end{remark}

\subsection{A periodic least-squares approximation problem}
We first consider the approximation of periodic functions on $[-T,T]$ using the periodized RBF. To that end we choose $N$ equispaced \emph{centers} $C := \{ c_j \}_{j=1}^N$,
\begin{equation}\label{eq:set_centers}
 c_j = -T + (j-1)\frac{2T}{N}, \qquad j=1,\ldots,N.
\end{equation}
For a given \emph{oversampling factor} $s \in \mathbb{N}^+$, we introduce a larger grid of $L = sN$ points $\tilde{X} := \{ \tilde{x}_i \} _{i=1} ^ L$,
\begin{equation}\label{eq:set_Xtilde}
 \tilde{x}_i = -T + (i-1)\frac{2T}{L}, \qquad i=1,\ldots,L.
\end{equation}
The periodic discrete least squares approximation problem in these points results in the rectangular matrix $A^{\rm c} \in \mathbb{R}^{L \times N}$ with entries given by
\begin{equation}\label{eq:matrix_entries}
    A^{\rm c}_{i j} = \phi_j( \tilde{x}_i), \qquad  i = 1,2,\cdots, L, \quad j = 1,2,\cdots,N.
\end{equation}
The basis functions are $\phi_j(x) = \phi^{\rm per}(x-c_j)$. For a given function $f$, smooth and periodic on $[-T,T]$, this results in the linear system
\begin{equation}\label{eq:circulant_system}
    A^{\rm c} \mb{a}^{\rm c} = \mb{b}^{\rm c},
\end{equation}
with right hand side $b_i^{\rm c} = f(\tilde{x}_i)$, $i=1,\ldots,L$. The solution of~\eqref{eq:circulant_system} corresponds to the approximation
\[
 f(x) \approx \sum_{j=1}^N a_j^{\rm c} \phi_j(x).
\]

The periodic structure of the problem is better expressed with a small change in formulation. Since the basis functions are periodic and the centers are equispaced, and owing to the choice of $L=sN$ as an integer multiple of $N$, every $s$ rows of $A^{\rm c}$ repeat in a circulant structure. We define a permutation matrix $\Pi$ that brings $A^{\rm c}$ into block-column\footnote{We use the terminology block-column circulant matrix here, rather than block-circulant matrix, to emphasize the structure of a column of blocks in which each block is circulant.} circulant structure:
\begin{equation}
    \Pi A^{\rm c} = 
    \begin{bmatrix}
    A^{\rm c}_1 \\ A^{\rm c}_2 \\  \vdots  \\ A^{\rm c}_s
    \end{bmatrix}.
\end{equation}
Here, $A^{\rm c}_k \in \mathbb{R}^{N \times N}$, with $k=1,\ldots,s$, is a circulant matrix that corresponds to the rows $k + (j-1)s$, with $j=1,\ldots,N$, of the original system.
Applying the same permutation $\Pi$ to the left and right hand sides of~\eqref{eq:circulant_system} yields
\begin{equation} \label{eq:circulant_system_pi}
    \Pi A^{\rm c} \mb{a}^{\rm c} = \Pi \mb{b}^{\rm c},
\end{equation}
in which the solution $\mb{a}^{\rm c}$ remains unchanged.
The permutation can be seen as the reordering of the $L=sN$ points of $\tilde{X}$ into $s$ shifted grids of length $N$ each.

\begin{remark}
 The motivation for considering uniform spacings both for the centers and the sample points is the possibility of using the FFT in an efficient implementation, as the FFT diagonalizes circulant matrices.
\end{remark}

\subsection{Non-periodic least-squares approximation problem}

A non-periodic function on an interval can be approximated using a periodic basis on a larger domain. For radial basis functions this strategy was proposed, analyzed and shown to be effective by C. Piret in~\cite{PiretFrames}. A study of convergence characteristics and stability is followed in~\cite{adcock2022rbf_frames}, {
and was mentioned when we introduced the choice of $c$ in~\eqref{eq:constant_shapeparam}.}
We adopt the setting of these references. Thus, we choose the approximation domain $\Omega = [-1,1]$, embedded in a bounding box $\Omega_{\rm B} = [-T,T]$ with $T>1$. Centers and over-sampled grid points are defined by~\eqref{eq:set_centers} and~\eqref{eq:set_Xtilde}. Collocation points are selected from $\tilde{X}$ as the subset in $[-1,1]$:
\begin{equation}\label{eq:set_X}
 X = \{ x \in \tilde{X} \, | \, x \in [-1,1] \}.
\end{equation}
We denote the cardinality of that set by $M = \#X$. We choose the oversampling factor $s$ large enough such that $M > N$. We have approximately $M \approx \frac{L}{T} = \frac{sN}{T}$, so that it is sufficient that $s > T$.

The centers, grid points and collocation points are illustrated in Fig.~\ref{fig:fig_colpts_1D}. 

\begin{figure}[H]
\centering
\includegraphics[width=\textwidth]{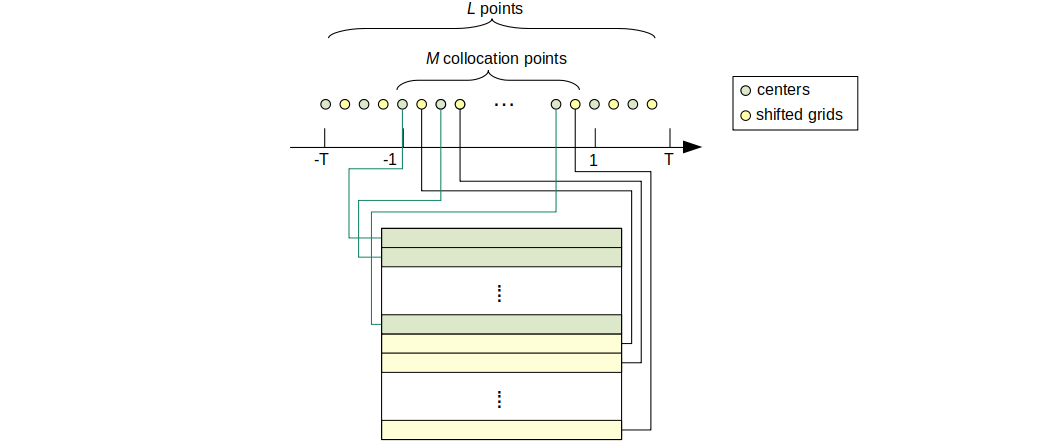}
\caption{Schematic description of the centers and collocation points for approximation on $[-1,1]$. The radial basis functions are periodized on the larger interval $[-T,T]$ and their centers are equispaced on $[-T,T]$ as well. The collocation points are equispaced on $[-1,1]$, in such a way that they form a subset of a larger equispaced grid on $[-T,T]$. Finally, oversampling by a factor $s$ is achieved by considering $s$ shifted grids. These grids are grouped together into the matrix $A$. The figure shows the case $s=2$.}\label{fig:fig_colpts_1D}
\end{figure}

Least squares approximation results in the rectangular matrix $A \in \mathbb{R}^{M \times N}$ with entries given by
\begin{equation}
    A_{i j} = \phi_j(x_i), \qquad  i = 1,2,\cdots, M, \quad j = 1,2,\cdots,N.
\end{equation}
For a given function $f$, letting $b_i = f(x_i)$, $i=1,\ldots,M$, we arrive at the linear system
\begin{equation}\label{eq:Ax_is_b}
    A \mb{a} = \mb{b}.
\end{equation}
Here, we order the collocation points $X$ in the same way as $\Pi \tilde{X}$, which is also illustrated in Fig.~\ref{fig:fig_colpts_1D}. This ensures that $A$ is embedded in the block-column circulant matrix $\Pi A^{\rm c}$, containing exactly a subset of its rows. We denote by $R \in \mathbb{R}^{M \times L}$ the corresponding restriction operator, i.e.,
\[
 A = R \Pi A^{\rm c}, \quad \mb{b} = R \Pi \mb{b}^{\rm c}.
\]
The restriction operator $R$ is not invertible, but its pseudo-inverse $E = R^\dagger$ corresponds to extension by zero-padding.

This method is illustrated with examples further on.

\begin{remark}
 When using Gaussian kernels, the periodicity of the extension does not affect the approximation space, at least not numerically. Indeed, as $N$ increases the kernels become increasingly narrow, and the periodized kernels near the edges of the bounding box are exponentially small on the domain of approximation. This effect plays a role in the proof of Theorem~\ref{theo:rank}.
\end{remark}

\section{Efficient algorithms for the approximation problems}
\label{ss:eff_alg_1d}

\subsection{Efficient algorithm for the periodic problem}
\label{ss:alg_per_1d}

A circulant matrix $C \in \mathbb{C}^{N \times N}$ is diagonalized using the discrete Fourier transform (DFT) matrix $F$ by $F^* C F$.
Here, $F \in \mathbb{C}^{N \times N}$ is given element-wise by
\[
 F_{kl} = \frac{1}{\sqrt{N}} {\rm e}^{2\pi {\rm i} (k-1)(l-1)/N}, \quad k,l = 1, \cdots, N,
\]
normalized so that $F F^* = I$. It can be applied efficiently using the fast Fourier transform (FFT).

Since $\Pi A^{\rm c}$ is block-column circulant, it is diagonalized by multiplication with $F$ on the right and with a block-diagonal matrix with $F^*$ on the diagonals on the left. E.g., for $s=2$, the latter matrix is
\begin{equation}\label{eq:P}
    P = \begin{bmatrix}
    F^* & 0\\
    0 & F^* 
    \end{bmatrix}.
\end{equation}
Left multiplying both sides of~\eqref{eq:circulant_system_pi} by $P$ leads to
\[
 P \Pi A^{\rm c} \mb{a}^{\rm c} = P \Pi \mb{b}^{\rm c}.
\]
Letting $\mb{y}^{\rm c} = F^* \mb{a}^{\rm c}$, and hence $\mb{a}^{\rm c} = F \mb{y}^{\rm c}$, we arrive at
\begin{equation}\label{eq:blockdiagonal_system}
 B \mb{y}^{\rm c} = \mb{c}^{\rm c},
\end{equation}
with $B = P \Pi A^{\rm c} F$ and $\mb{c}^{\rm c} = P \Pi \mb{b}^{\rm c}$. The matrix $B$ is block-column diagonal, consisting of diagonal blocks $D_i$, with diagonals $\mb{d}_i$, $i = 1, \ldots, s$. For example, in case of the Gaussian kernel the entries of $\mb{d}_i$ are given explicitly by
\begin{align} \label{eq:explicit_d}
    \left( \mb{d}_i \right)_k
    &= \sum_{n=1}^N \left[ \sum_{m=-\infty}^\infty {\rm e}^{-\varepsilon^2 ((n-1)h + (i-1)h_L + 2mT)^2} \right] {\rm e}^{\frac{{\rm i}2\pi(k-1)(n-1)}{N}} \notag \\
    &= \sum_{n=-\infty}^{+\infty} {\rm e}^{-\varepsilon^2((n-1)h + (i-1)h_L)^2} {\rm e}^{\frac{{\rm i}2\pi(k-1)(n-1)}{N}}.
\end{align}
The second line also shows that the infinite summation over samples of the kernel $\phi$ is equivalent to a finite sum over its periodization. We will use this property in the theory later on.

The least squares system~\eqref{eq:blockdiagonal_system} is solved by $\mb{y}^{\rm c} = B^\dagger\mb{c}^{\rm c}$, where the  pseudo-inverse of $B$ is block-row diagonal.  The blocks $\tilde{D}_i$ of $B^\dagger$ are given by

\begin{equation}
\tilde{D}_i =\left( \sum_{j=1}^s D_j^* D_j \right)^{-1}  D_i^* .
\end{equation}
For example, when $s=2$ we have
\[
 B = 
\begin{bmatrix}
D_1 \\ D_2
\end{bmatrix},
\]
and
\begin{equation}\label{eq:B_pinv}
 B^\dagger = (B^* B)^{-1} B^* = \begin{bmatrix}  \tilde{D}_1 & \tilde{D}_2  \end{bmatrix}.
\end{equation}
Here, the diagonals $\tilde{\mb{d}}_{1/2}$ of $\tilde{D}_{1/2}$ are given element-wise in terms of the diagonals $\mb{d}_{1/2}$ of $D_{1/2}$ by
\[
\tilde{\mb{d}}_{1,j} = \frac{\overline{\mb{d}}_{1,j}}{|\mb{d}_{1,j}|^2 + |\mb{d}_{2,j}|^2} \quad \mbox{and} \quad \tilde{\mb{d}}_{2,j} =  \frac{\overline{\mb{d}}_{2,j}}{|\mb{d}_{1,j}|^2 + |\mb{d}_{2,j}|^2}.
\]

Thus,~\eqref{eq:blockdiagonal_system} can be solved efficiently using the pseudo-inverse. This is followed by a single FFT
\[
 \mb{a}^{\rm c} = F \mb{y}^{\rm c},
\]
to recover the least squares solution of the original block-column circulant system.\footnote{The solution of a system involving a circulant matrix using the FFT is a standard operation. The authors are unaware of an extension of that approach in literature to block-column circulant matrices as outlined here.}

\subsection{AZ algorithm for non-periodic least-squares approximation problem}

\subsubsection{AZ algorithm}
The AZ algorithm can be used to speed up the solution of least squares problems in cases where an incomplete generalized inverse is known~\cite{coppe2020az}. In this section we show how to apply the AZ algorithm to linear system~\eqref{eq:Ax_is_b}, based on an efficient solver for the related block-column circulant system~\eqref{eq:circulant_system_pi}.  Other example applications of AZ include its original use to solve 2D Fourier extension problem~\cite{matthysen2017fastfe2d}, an application to spline-based approximations in~\cite{coppe2022splines} and enriched approximation spaces in~\cite{herremans2023az}.

Given a linear system $A \mb{x}=\mb{b}$ with $A \in \mathbb{C}^{M \times N}$, and an additional matrix $Z \in \mathbb{C}^{M \times N}$, the AZ algorithm consists of three steps.
\begin{algorithm}[H]
\caption{AZ algorithm}\label{alg: az_1d}
\textbf{Input:} $A, Z \in \mathbb{C}^{M \times N}$, $\mb{b} \in \mathbb{C}^M$\\
\textbf{Output:} $\mb{x} \in \mathbb{C}^N$ such that $A \mb{x} \approx \mb{b}$ in a least squares sense
\begin{algorithmic}[1]
\State Solve $(I - AZ^*)A \mb{x}_2 = (I - AZ^*) \mb{b}$
\State $\mb{x}_1 \leftarrow Z^* (\mb{b} - A \mb{x}_2)$
\State $\mb{x} \leftarrow \mb{x}_1 + \mb{x}_2$
\end{algorithmic}
\end{algorithm}

The algorithm crucially relies on the choice of a suitable matrix $Z^*$. That matrix is characterized in~\cite{coppe2020az} as an \emph{incomplete generalized inverse}. A generalized inverse $A^g$ of $A$ is such that
\[
 A A^g A = A \quad \iff A A^g A - A = 0.
\]
A matrix may have many generalized inverses, especially if it is rectangular or not of full rank~\cite{ben2003generalized,golub2013matrix}. Ideally, for the application of AZ, matrix $Z$ is close to a generalized inverse of $A$ in the sense that $A Z^*A -A$ has low rank. Thus, the system to be solved in step 1 of the AZ algorithm has low rank. In combination with a fast matrix-vector product for both $A$ and $Z^*$, the solver is amenable to randomized linear algebra methods~\cite{halko2011finding}. Step 2 consists of matrix-vector products with $A$ and $Z^*$ only, hence the method is efficient overall.

\subsubsection{Non-periodic least-squares approximation problem}
Our goal is to show that the solver of the periodic approximation problem leads to a suitable choice of $Z$ in the AZ algorithm for the non-periodic problem.

Since the matrix $A$ can be decomposed as $A = R \Pi A^{\rm c} = R P^* B F^*$ and we have already calculated the inverse or pseudo-inverse of these matrices explicitly, we can intuitively choose the incomplete generalized inverse matrix $Z^*$ as 
\[
Z^* = F B^\dagger P E.
\]
Thus the left hand side of the first step of AZ algorithm becomes
\[
A - A Z^* A = R P^* B F^* - (R P^* B F^*) (F B^\dagger P E) (R P^* B F^*).
\]
We prove later on that this matrix has rank much lower than $A$.

\subsubsection{Efficient implementation and computational complexity}
\label{ss:alg_eff_1d}

The first step of Algorithm~\ref{alg: az_1d} involves the solution of a low-rank linear system and some matrix-vector products. If these matrix-vector products can be computed efficiently, then the system of step 1 can be solved efficiently as well using randomized linear algebra~\cite{coppe2020az}.
We consider the different ingredients separately.

\paragraph{Multiplication with $B$ and $B^\dagger$}
$B$ and $B^\dagger$ consists of $s$ diagonal matrices, hence the cost of computing $B\mb{v}$ and $B^\dagger \mb{v}$ is $\mathcal{O}(s N)$. 

\paragraph{Multiplication with $R$ and $E$}
The operation $ER$ formally corresponds to setting rows to zero, and $R$ corresponds to selecting some rows of a matrix. These operations have no corresponding computational cost, as the rows are simply omitted in the calculations.

\paragraph{Multiplication with $P$ and $F$}
As mentioned before, multiplied by $F$ and $P$ costs the same as doing FFT, which is $N \log N$ and $s N \log N$.

\paragraph{Low-rank solver for $Ax=b$.} If matrix $A \in \mathbb{C}^{M\times N}$ has rank $r$, then the randomized low-rank solver requires $r+p$ matrix vector products, where $p$ is a small number to ensure that with high probability the randomized matrix has at least rank $r$. This is followed by the SVD of an $(r+p)$-by-$M$ matrix. The latter takes ${\mathcal O}(r^2 M)$ operations.

\paragraph{Computational complexity of the AZ solver}
With $A\in \mathbb{C}^{M\times N}$ and $M = {\mathcal O}(N)$, a direct solver for a rectangular least-squares problem $Ax=b$ scales as $\mathcal{O}(MN^2) = \mathcal{O}(N^3)$. For the AZ algorithm, the computational complexities of the steps are:
\begin{itemize}
    \item Step 0: Preparations
        \begin{itemize}
            \item Computation of $B$: $\mathcal{O}(sN \log N)$ flops.
            \item Computation of $B^\dagger$: $\mathcal{O}(sN)$ flops.
        \end{itemize}
    \item Step 1: 
        \begin{itemize}
            \item Computation of system matrix: $\mathcal{O}(sN \log N)$ flops.
            \item Computation of right-hand side: $\mathcal{O}(sN \log N)$ flops.
            \item Solving the least-squares system: if the rank of $A - AZ^*A$ is $r$, ${\mathcal{O}}(rsN\log N)$ flops for the matrix-vector products and ${\mathcal{O}}(M r^2)$ flops for the SVD.
        \end{itemize}
    \item Step 2: The computation of matrix-vector products $\mathcal{O}(sN \log N)$ flops.
    \item Step 3: The addition takes $\mathcal{O}(M)$ flops.
\end{itemize}
The dominant term depends on $r$, $s$ and $N$. Since $s$ is always very small and fixed, the complexity of AZ solver scales as $\mathcal{O}(r N \log N)$. If in addition $r$ is bounded independently of $N$, the total cost is just a multiple of $r$ of the FFT. In the later section \S~\ref{ss:theo} we will show precisely that, namely that $r$ is bounded independently of $N$.

\subsection{Numerical results}
We illustrate the effectiveness and performance of the method with a number of examples.

We illustrate the periodic solver first, in comparison to backslash for the least squares system. Fig.~\ref{fig:d1_fa_per} shows the results of approximating the periodic function $f(x) = \sin \left( \left[ \frac{N}{5} \right] \pi x \right) $ with periodized basis functions in $[-1,1]$.
The left panel shows that the maximum error of both solvers are below $1e-12$. The right panel shows that the computing time of backslash scales like $\mathcal{O}(N^3)$ while the computing time of the efficient solver is only $\mathcal{O}(N \log N)$. Though not shown in the figure, we also remark that the coefficient norms of the solutions in both cases remains bounded, with $\frac{\lVert \mb{a} \rVert}{\sqrt{N}} \approx 5$, suggesting both solvers are numerically stable.

Fig.~\ref{fig:d1_fa_rank} shows the rank and singular values of $A$ and $A-AZ^*A$. The left panel shows their rank. The rank of  $A-AZ^*A$ is a constant irrespective of the size of $N$, which will be explained in the next section. The right panel shows the singular values of $A-AZ^*A$ and $A$ for some specific $N$. The spectrum of $A$ has a continuous part, followed by a rapid plunge towards zero. The matrix $A-AZ^*A$ isolates this plunge region of size about $10$, which is why the AZ algorithm is efficient.

Fig.~\ref{fig:d1_fa_nonper} shows the results of approximating the non-periodic function $f(x) = \sin \left(  \frac{N}{5}  x \right) $ in $[-1,1]$ with periodic basis functions in $[-1.5,1.5]$ and oversampling factor $s=3$. The figure shows favourable complexity of the AZ algorithm. Note that the function becomes increasinlyg oscillatory with increasing $N$, hence the difficulty of the approximation problem increases with $N$ as well. The results nevertheless show that the error remains bounded. The coefficient norms of the solutions using these two solvers are also bounded, with $\frac{\lVert \mb{a} \rVert}{\sqrt{N}} < 1$ in this case. Since the main cost of the algorithm is a number of FFTs, the AZ solver scales very well to large $N$. On a contemporary laptop we observed that the computing time remains under $10$ seconds up to a million degrees of freedom ($N = 10^6$).

\begin{figure}[H]
\centering
\includegraphics[width=\textwidth]{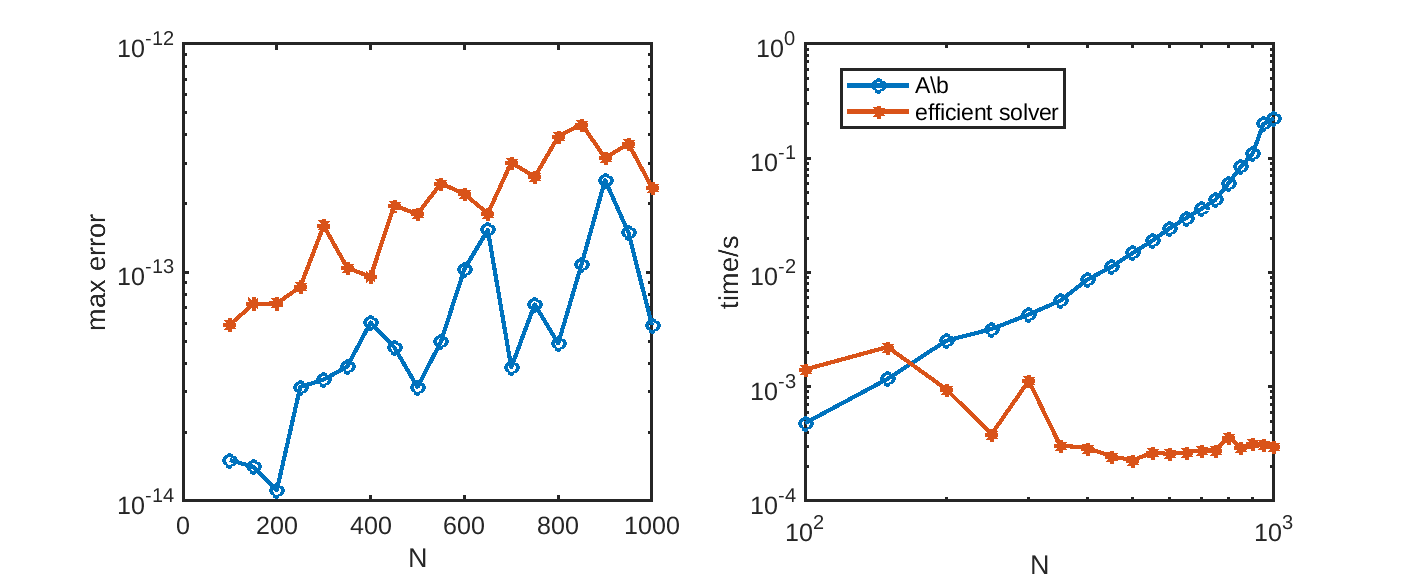}
\caption{We approximate the periodic function $f(x) = \sin \left( \left[ \frac{N}{5} \right] \pi x \right) $, with $T=1$ and $s=3$. Although the function is simple, it becomes increasingly oscillatory with $N$ and, hence, the approximation difficulty is constant. The purpose is to illustrate stability of the method for large $N$, which is confirmed in the left panel: the error remains roughly constant as $N$ increases. The right panel shows the computing time. The efficient solver outperforms a direct solver once $N > 150$.
}\label{fig:d1_fa_per}
\end{figure}

\begin{figure}[H]
\centering
\includegraphics[width=\textwidth]{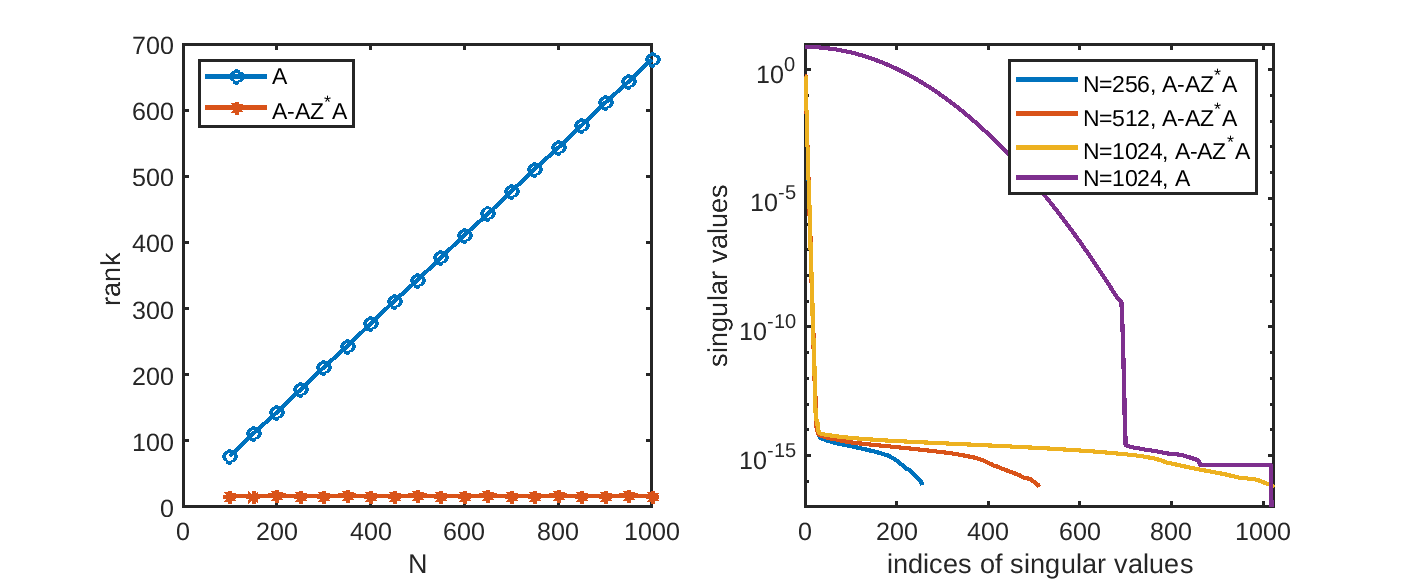}
\caption{Comparison of the rank of system matrix $A$ and $A - AZ^*A$. The left panel shows the rank of the two matrices with different $N$. The rank of $A$ increases with $N$, while the rank of  $A-AZ^*A$ remains constant.
The right panel shows the singular values of the two matrices with different $N$. The spectrum of $A-AZ^*A$ consists only of a so-called plunge region up to $N=10$, which is why the AZ algorithm is efficient.}\label{fig:d1_fa_rank}
\end{figure}

\begin{figure}[H]
\centering
\includegraphics[width=\textwidth]{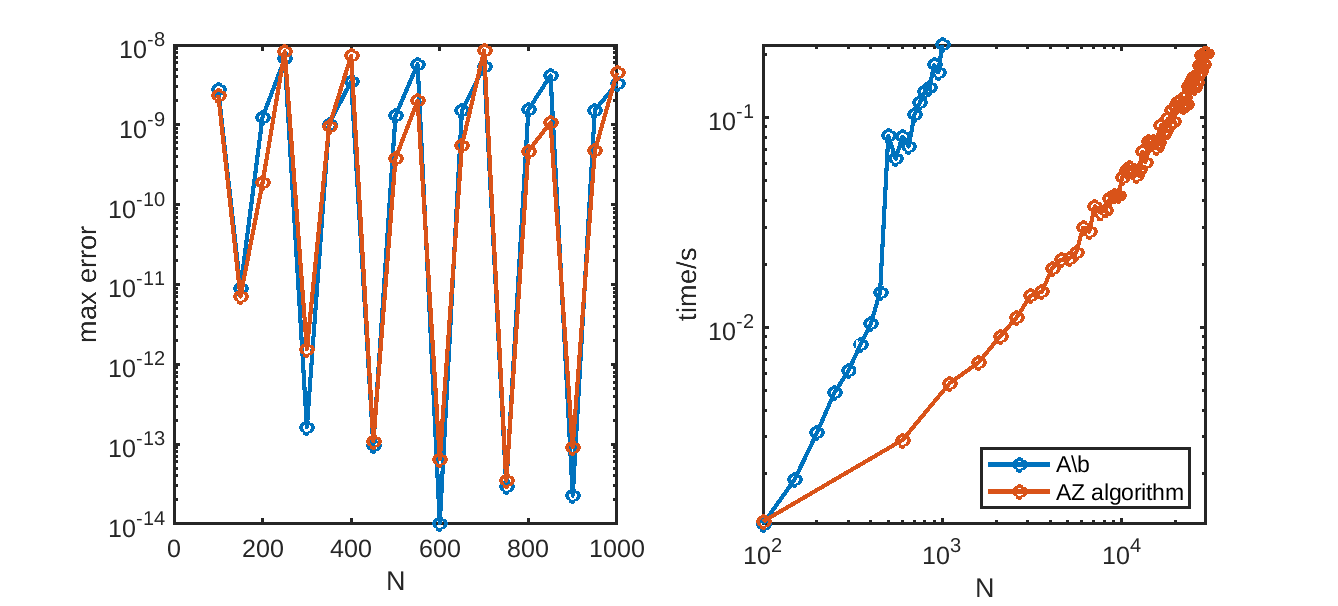}
\caption{We approximate the non-periodic function: $f(x) = \sin \left(  \frac{N}{5}  x \right) $ in $[-1,1]$, with $T=1.5$ and $s=2$. This function becomes increasingly oscillatory with $N$. The purpose is to illustrate stability of the method for large $N$, which is confirmed in the left panel: the error remains roughly constant as $N$ increases. The right panel shows the computing time. The AZ algorithm exhibits near linear complexity. }\label{fig:d1_fa_nonper}
\end{figure}

\section{Theoretical considerations}
\label{ss:theo}
We analyze the accuracy and efficiency of the AZ algorithm for univariate RBF approximations.

\subsection{Accuracy of AZ}
We note that, by construction, $Z^*$ is an exact solver for the periodic approximation problem. However, since matrix $A$ is typically ill-conditioned, owing to small singular values, matrix $Z^*$ necessarily has a large norm. This means that computations involving $Z^*$ may lead to loss of numerical accuracy. In this section we quantify the size of $\Vert Z^*\Vert$ and analyze its effect.

The accuracy of the AZ algorithm is described by the following Lemma, which we quote from~\cite{coppe2020az}.

\begin{lemma}[{\cite[Lemma 2.2]{coppe2020az}}]\label{theo:azerror}
Let $A \in \mathbb{R}^{M\times N}$, $\mathbf{b} \in \mathbb{R}^M$, and suppose there exists a \emph{stable least squares fit} $\tilde{x} \in \mathbb{R}^N$ in the sense that
\begin{equation}\label{eq:stablefit}
 \|b-A\tilde{x}\| \leq \eta, \qquad \|\tilde{x}\| \leq C,
\end{equation}
for $\eta,\, C > 0$. Then there exists a solution $\hat{x}_1$ to step 1 of the AZ algorithm such that the computed solution $\hat{x} = \hat{x}_1 + \hat{x}_2$ satisfies,
\begin{equation} \label{eq:norm_AZ_lemma}
 \|b-A \hat{x}\| \leq \|I-AZ^*\| \eta, \qquad \|\hat{x}\| \leq C + \|Z^*\| \eta.
\end{equation}
\end{lemma}

The existence of a stable least squares fit to the RBF approximation problem is guaranteed, in the case of the Gaussian RBF, by the analysis in~\cite[\S4]{adcock2022rbf_frames}. Thus, the previous lemma implies accuracy of the AZ solution subject to bounds on $\Vert I - AZ^*\Vert$ and $\Vert Z^* \Vert$.

\begin{lemma}\label{lem:dft}
Let $h_L = \frac{2T}{L}$ and $h_N = \frac{2T}{N}$ and define $s$ shifted grids $\mathbf{\tilde{x}}_i$ of length $N$ by
\[ 
\tilde{x}_{i,k} = -T + (i-1) h_L + (k-1) h_N, \qquad i=1,\ldots,s, \quad k=1,\ldots,N.
\]
The diagonals $\mathbf{d}_i$ of the block-column diagonal matrix $B$ in~\eqref{eq:blockdiagonal_system} are the DFT of the periodic kernel $\phi_1(x) = \phi^{\rm per}(x-c_1)$ evaluated in these shifted grids. That is, $\mathbf{d}_i = F^* \mathbf{v}_i$ with $v_{i,k} = \phi_1(\tilde{x}_{i,k})$.
\end{lemma}
\begin{proof}
This follows by construction of $B$. Note that the grid of length $L = sN$ is the union of $s$ shifted grids as defined in this lemma. Hence, the first $s$ rows of $A^{\rm c}$, defined by~\eqref{eq:matrix_entries} are the first rows of each circulant block in $\Pi A^{\rm c}$. The diagonals of $B$ are the DFT of these rows.
\end{proof}

\begin{lemma}
\label{theo:tau}
For the Gaussian RBF $\phi(r) = e^{-\varepsilon^2 r^2}$ with $\varepsilon = c^*N$ chosen according to~\eqref{eq:constant_shapeparam} for a small value of $\tau_0$, the maximal entry of the block-row diagonal matrix $Z^*$ is
\[
 (Z^*)_{N/2+1,N/2+1} = {\mathcal O}(\tau_0^{-1}).
\]
\end{lemma}
\begin{proof}
By Lemma~\ref{lem:dft} the diagonals of $B$ are given by the discrete Fourier transform of the kernel, sampled at equispaced points. For the case of the Gaussian kernel, we can derive an explicit expression. 

The Fourier transform of $x(t) = {\rm e} ^{-\varepsilon^2 h^2 t^2}$ is $X(k) = \frac{\sqrt{\pi}}{\varepsilon h} \exp \left( - \frac{\pi^2 k^2}{\varepsilon^2 h^2} \right)$. By the Poisson summation formula, for certain sample period $T$ and sample rate $k_s = 1/T$, the samples $x(nT)$ correspond to a periodic summation of $X(k)$:
\[
{\displaystyle X_{s}(k)\ \triangleq \sum _{m=-\infty }^{+\infty } X\left(k-m k_{s}\right)=\sum _{n=-\infty }^{+\infty }T\cdot x(nT) {\rm e}^{-i2\pi nTk}}.
\]

In our setting $T=1$ and $x(nT)  = x(n) = {\rm e} ^{ - \varepsilon^2 h^2 n^2}  $. The corresponding sample rate is $k_{\rm s} = 1/T = 1$. This leads to
\[ 
    X_{\rm s} (k) = \sum_{m = -\infty}^{+\infty} X(k - m)
    =  \sum_{n=-\infty}^{+ \infty}  x(n) {\rm e}^{-i2\pi nk}.
\]
After a scaling with $N$, the number of centers,
\begin{align*}
    X_{\rm s} \left(\frac{k}{N}\right) &= \sum_{m = -\infty}^{+\infty} X(k/N - m) 
    = \sum_{n=-\infty}^{+\infty} x(n) {\rm e}^{-i2\pi n k/N} \\
    &= \sum_{n=-\infty}^{+\infty} {\rm e}^{-\varepsilon^2(n h)^2} {\rm e}^{{\rm i}2\pi n k /N}
    = \left( \mb{d}_1 \right)_{k+1}.
\end{align*}
The last equality is the same as~\eqref{eq:explicit_d}. We see that $X_{\rm s} (\frac{k-1}{N})$ is equal to $\left( \mb{d}_1 \right)_k$, the first diagonal of $B$. It is a sum of overlapping Gaussian's, for which the minimum value arises at the point in the middle of two centers, corresponding to index $k=N/2$. This leads to
\begin{align*}
    X_{\rm s} \left( \frac{1}{2} \right) &=   \frac{\sqrt{\pi}}{\varepsilon h} \sum_{m = -\infty}^{+\infty} \exp \left( - \frac{\pi^2 (1/2-m)^2}{\varepsilon^2 h^2} \right)
    \\
    & \approx  2 \sqrt{ \frac{2 \log \left( 1+\tau_0^{-2} \right)}{\pi}} \sum_{m = 0}^{+\infty} \tau_0 ^{ 4 (1/2-m)^2} = \mathcal{O}(\tau_0).
\end{align*}
This is also $\left( \mb{d}_1 \right)_{N/2+1}$, the minimum value of $\mb{d}_1$.

The expressions for  other diagonals $\mb{d}_i$, $i=2,\ldots,s$, are entirely similar, as a shift in time merely corresponds to a phase shift in the frequency domain. However, the elements of the other diagonals might be smaller, as the phase shifts may induce terms with alternating signs, whereas the Gaussian kernel and its translates in the expression for $\mb{d}_1$ themselves are always positive. Since the diagonals of $Z^*$ are given by the inverse of the sum of the squares of all diagonals, the maximal entry of $Z^*$ is determined by the smallest element of $\mb{d}_1$. Hence it is $\mathcal{O}(\tau_0^{-1})$.
\end{proof}

The result is that the norm of $Z$ may indeed be large, but it is bounded. This happens in spite of the exponential decay of the Fourier transform of the Gaussian function. Loosely speaking, that decay is counteracted by the scaling with $N$, which causes the support of the kernels to decrease (and their Fourier transforms to grow) as $N$ increases. Both effects balance each other.

Since $\Vert Z^* \Vert$ multiplies $\eta$ in~\eqref{eq:norm_AZ_lemma}, its impact still remains bounded if $\eta \sim \tau_0$. Moreover, we can also bound the other factor $\Vert I - AZ^*\Vert$ in our setting.

\begin{lemma}
\label{theo:IminAZ}
\[
\Vert I - A Z^* \Vert \leq 2.
\]
\end{lemma}

\begin{proof}
Since $A$ and $Z^*$ can be written as $RP^*BF^*$ and $FB^\dagger PE$ separately, 
\[
I - A Z^* = I - R P^* B F^* F B^\dagger P E = I - A Z^* = I - R P^* B B^\dagger P E ,
\]
where $B B^\dagger$ is an orthogonal projection, by properties of the pseudo-inverse, $P$ and $P^*$ are unitary matrices, $E$ and $R$ are also orthogonal projection, hence so is $ R P^* B B^\dagger P E 
 $. According to triangle inequality,
 \[
 \Vert I - A Z^* \Vert \leq  \Vert I \Vert  +  \Vert A Z^* \Vert \leq 2.
 \]

\end{proof}

We arrive at the following result.

\begin{theorem}\label{theo:accuracy}
If the discrete least squares problem~\eqref{eq:Ax_is_b} using the periodized Gaussian RBF has a solution satisfying
\begin{equation*}\label{eq:stablefit}
 \|b-A\tilde{x}\| \leq C \tau_0, \qquad \|\tilde{x}\| \leq D,
\end{equation*}
for some constants $C, D > 0$, and with $\tau_0$ determining the shape parameter $\varepsilon$, then a solution $\mathbf{x}$ of the AZ algorithm exists that satisfies
\[
\Vert A \mathbf{x} -b \Vert\leq C \tau_0 \quad \mbox{and} \quad \Vert \mathbf{x} \Vert \leq E,
\]
for some constant $E > 0$ independently of $\tau_0$.
\end{theorem}
\begin{proof}
It follows from Lemma~\ref{theo:azerror} by combining Lemma~\ref{theo:tau} and Lemma~\ref{theo:IminAZ} that a solution exists with the given bound.
\end{proof}
The importance of establishing the existence of stable least squares fits is that such solutions can be computed using a truncated SVD~\cite{adcock2020frames,coppe2020az} and, with high probability, using randomized linear algebra methods~\cite[Theorem 3.4]{coppe2020az}. Thus, in spite of large norm of $Z$, the AZ algorithm for univariate RBF approximation is numerically stable and accurate.

As we shall see further on, the conclusion is different for 2D approximations.
Different from Lemma~\ref{theo:tau}, the maximal entry of $Z^*$ is $\mathcal{O} (\tau_0^{-2})$, but it is also bounded. It is because the matrix $Z^*$ in 2D is the tensor product of incomplete generalized inverse matrices in two directions, and the maximal entries are both in the middle of diagonals.
Also according to~\eqref{eq:norm_AZ_lemma}, the influence of  $\Vert Z^* \Vert$  still remains bounded if $\eta \sim \tau_0^2$.

\subsection{Efficiency of AZ}
The efficiency depends on the rank $r$ of the system in step 1. Recall that the computational cost of the SVD in this step is ${\mathcal O}(Mr^2)$. Using the fact that the Gaussian RBF nearly has compact support, because it decays so quickly, we can establish that the rank $r$ is essentially constant, making this step ${\mathcal O}(M)$ and, with a constant oversampling factor, also ${\mathcal O}(N)$.

To that end we introduce a truncation parameter $\delta$ in the following theorem, below which we truncate the Gaussian basis functions. Although the proof of the statement is technically involved, the gist of it is that there is just a bounded number of truncated basis functions that overlap with the boundary of the approximation domain. All other basis functions are either contained in the interior of the interval $[-1,1]$, or in the complement $[-T,T] \setminus [-1,1]$. Loosely speaking, those are accounted for exactly by the action of $Z^*$.

\begin{theorem}
\label{theo:rank}
Using Gaussian RBFs with shape parameter $\varepsilon$ satisfying~\eqref{eq:constant_shapeparam}, the linear system $A  - A  Z^* A$ has approximately constant rank in the following sense. There exists a bounded matrix $Q$, with $\Vert Q \Vert_\infty \leq 1$, such that for any small tolerance $\delta > 0$ we have
\[
 \RANK (A  - A  Z^* A + \delta Q)\leq 4W
\]
with
\begin{equation}
 W = \frac{1}{\pi} \sqrt{-2 \log \delta \log(1+\tau_0^{-2})}.
\end{equation}
\end{theorem}

\begin{proof}
First, we turn the system $A - AZ^*A$ into a similar system with block structure. We define $B^{\rm e} = P EAF^*$ and $\tilde{Z} = B^\dagger$, with $B^\dagger$ defined in~\eqref{eq:B_pinv}. Recall that $B = P \Pi A^{\rm c} F$ is block-column diagonal and $B^\dagger$ is its pseudo-inverse. One can think of the matrix $B^{\rm e}$ as an approximation of $B$ and, hence, $\tilde{Z}^*$ as an approximate pseudo-inverse of $B^{\rm e}$.

Using the definitions of the operators involved we find that
\[
B^{\rm e} - B^{\rm e} \tilde{Z}^* B^{\rm e} = PE(A - AZ^*A)F.
\]
Since $P$ and $F$ are full rank, and since $E$ simply extends the matrix by zeros without increasing its rank, we find that the rank of $A-AZ^*A$ and that of $B^{\rm e} - B^{\rm e} \tilde{Z}^* B^{\rm e}$ is exactly the same -- but the latter has a simpler structure to analyze.

Denote the diagonal blocks of $B$ by $D_i$, $i=1,2,\dots,s$, with corresponding diagonal vectors $\mb{d}_i$ given by~\eqref{eq:explicit_d}.  $B^{\rm e}$ can be written as 
\begin{align*}
    B^{\rm e} &= P A^{\rm e} F 
    = 
    \begin{bmatrix}
    F^* (E_1 R_1 A^{\rm c}_1) F \\
    F^* (E_2 R_2 A^{\rm c}_2) F \\
    \vdots \\
    F^* (E_s R_s A^{\rm c}_s) F \\
    \end{bmatrix}
    =
    \begin{bmatrix}
    C D_1 \\ C D_2 \\ \vdots \\ C D_s
    \end{bmatrix}.
\end{align*}
Here, $E_i \in \mathbb{C}^{N \times M_i}$ and $R_i \in \mathbb{C}^{M_i \times N}$, $i=1,2,\cdots,s$, are the matrices for extending and restricting $A_i^{\rm c}$ according to the subset of shifted centers in $[-1,1]$. Without loss of generality, all the $E_i$ are the same, and all the $R_i$ are the same.
We choose $\tilde{Z}^* = B^\dagger = \begin{bmatrix}
\tilde{Z}_1^* &  \tilde{Z}_2^* & \cdots &  \tilde{Z}_s^*  \end{bmatrix}$ for the AZ algorithm. Then $\tilde{Z}^* B = \sum_{i=1}^s \tilde{Z}^*_i B_i$ is an identity matrix.

Define $C = F^* E_i R_i F$, then we get 
\begin{align*}
     B^{\rm e}  - B^{\rm e}  \tilde{Z}^* B^{\rm e} &= 
    \begin{bmatrix}    
    C D_1 - \sum_{i=1}^s C D_1 \tilde{Z}_i^* C D_i \\ 
    C D_2 - \sum_{i=1}^s C D_2 \tilde{Z}_i^* C D_i \\ 
    \vdots \\ 
    C D_s -  \sum_{i=1}^s C D_s \tilde{Z}_i^* C D_i  
    \end{bmatrix},
\end{align*}
with the $j$th block 
\begin{align*}
     C D_j - \sum_{i=1}^s C D_j \tilde{Z}_i^* C D_i 
     &= C D_j - \left(  C D_j \tilde{Z}_j^* C D_j + \sum_{\substack{i=1\\ i\neq j}}^s C D_j \tilde{Z}_i^* C D_i  \right) \\
     &= - \left( \sum_{\substack{i=1\\ i\neq j}}^s \left(  C \tilde{Z}_i^* \left( D_j  C D_i -  D_i  C Dj \right)  \right) \right).
\end{align*}
The rank depends on $D_j  C D_i -  D_i  C D_j$.
The diagonals of $D_i$ only differ by a constant factor, thus we may proceed with $D_1 C D_2 - D_2 C D_1$:
\begin{align*}
    D_2 C D_1 - D_1 C D_2 
    &= F^* (A^{\rm c}_2 E_1 R_1 A^{\rm c}_1 - A^{\rm c}_1  E_2 R_2 A^{\rm c}_2 ) F.
\end{align*}
We define $Q = A^{\rm c}_2 E_1 R_1 A^{\rm c}_1 - A^{\rm c}_1  E_2 R_2  A^{\rm c}_2$. Without loss of generality, the operator $E_i R_i$ selects the $N_1$-th row until the $N_2$-th row of $A^{\rm c}_i$ and sets the other rows to zeros, with $1 \leq N_1 < N_2 \leq N$, which corresponds to selecting the rows according to the subset of centers in $[-1,1]$. Now we consider the properties of matrix $Q$,
\begin{align*}
    Q_{i,j}  
    = \sum_{k=N_1}^{N_2} \left[ \phi_k \left( c_i + 2T/L \right)  \phi_j \left( c_k \right) - \phi_k \left( c_i \right)   \phi_j \left( c_k + 2T/L \right) \right] .
\end{align*}
Recall that the centers $ \lbrace c_j \rbrace _{j=1}^N $ are defined in~\eqref{eq:set_centers} and $2T/L$ is the distance between two neighboring grid points.

At this point, it becomes clear that the numerical support of the Gaussian RBFs determines the numerical rank of the system. The function value $\phi_i (c_j) = {\rm e}^{-\varepsilon^2 (c_j - c_i)^2}  = {\rm e}^{ - \frac{4 \varepsilon^2 T^2}{N^2} (j-i)^2 }$ decays when the difference of $i$ and $j$ becomes larger. Since $\varepsilon T/N$ is a constant, the value of $\phi_i (c_j) $ becomes smaller than the threshold $\delta$ as soon as
\[
\left| j-i  \right| = \sqrt{-2\log \delta \log (1+\tau_0^{-2})}/\pi := W.
\]
Thus the elements in $Q$ are significant (with respect to the threshold) only when they are near the diagonals within bandwidth $W$.

For the elements in the middle square of the matrix, $i \in [N_1 + W +1, N_2 - W]$ or $j \in [N_1 + W +1, N_2 - W]$, the elements of $Q$ are all smaller than the threshold. 
Because
\[
Q_{i,j} = \sum_{k=i- W}^{j + W} \left[  (A_2^{\rm c})_{i,k} (A_1^{\rm c})_{k,j} -  (A_1^{\rm c})_{i,k} (A_2^{\rm c})_{k,j}  \right],
\]
the elements in the brackets are odd symmetric with respect to $k = (i+j)/2$, which means they sum to zero. This symmetry argument can not be used near the boundary of the matrix $Q$.

As a result, there are only two non-negligible subblocks of $Q$ on the northwest and southeast corners and they both have size $2W$-by-$2W$. Thus the rank of $B^{\rm e} - B^{\rm e} Z^* B^{\rm e}$ is at most $4W$.
\end{proof}

\section{2D problems}
\label{ss:2d_fa}
The extension of the univariate scheme to problems in two dimensions is reasonably straightforward. We focus on the approximation on a domain $\Omega$ that can be embedded into a bounding box
\[
\Omega \subset \Omega_{\rm B}= [-T^x, T^x] \times [-T^y, T^y].
\]
We consider periodized RBFs on the box to approximate a function on $\Omega$. An important feature is that we do not assume $\Omega$ itself to be rectangular. Its shape can be arbitrary.

Radial basis functions in 2D are typically defined in terms of radial distances. The Gaussian RBF is unique in the property that its radial 2D generalization has tensor product structure, $\phi(\| \mathbf{x} - \mathbf{y} \|) = \phi( \| x_1 - y_1 \| )  \phi( \| x_2 - y_2 \| )$. For this reason, we restrict this section to the Gaussian RBF.

\subsection{Periodic approximation problem and efficient solver}
\label{ss:alg_per_2d}

The periodic approximation problem is defined on the box $\Omega_{\rm B}$ and it has full product structure, which leads to similar formulas as in the 1D case. Let us establish notations. We are led to solve the linear system
\begin{equation}\label{eq:circulant_system_2d}
    A^{\rm c} \mb{a}^{\rm c} = \mb{b}^{\rm c},
\end{equation}
in which $A^{\rm c} = A^{{\rm c}x} \otimes  A^{{\rm c}y}$ is the product of two block-column circulant matrices. The latter is based on using centers $C^x \times C^y$, with $N^x$ and $N^y$ equispaced centers in the $x$ and $y$ direction on $\Omega_B$ respectively. The sampling points are $\tilde{X} \times \tilde{Y}$, with $L^x = s^x N^x$ and $L^y = s^y N^y$ points in the $x$ and $y$ directions. We call the integers $s^x$ and $s^y$ the \emph{oversampling factors}.

This results in an approximation of the form
\begin{equation}
    f(x,y) \approx \sum_{j=1}^{N_x N_y} a_j \phi_j(x,y),
\end{equation}
in which the basis functions $\phi_j (x,y) = \phi_m^x (x) \phi_n^y (y)$ are products of univariate periodized functions, with $j = (m-1)N^y +n$ for $m=1,\cdots, N^x$ and $n = 1,\cdots, N^y$. These univariate functions are associated with the centers in $C^x \times C^y$,
\[
 \phi^x_m(x) = \phi^{\rm per}( \lVert	x - c_m^x \rVert ), \quad  \phi^y_n(y) = \phi^{\rm per}( \lVert y - c_n^y \rVert ).
\]

Similar to the univariate problem, we block-diagonalize the system. With block-diagonal matrices $P^x$ and $P^y$, with DFT matrices $F^x$ and $F^y$ along their diagonals, we obtain
\[
 (P^x \otimes P^y) A^{\rm c} \mb{a}^{\rm c} = (P^x \otimes P^y) \mb{b}^{\rm c}.
\]
Letting $\mb{y}^{\rm c} = (F^x \otimes F^y)^* \mb{a}^{\rm c}$, and hence $\mb{a}^{\rm c} = (F^x \otimes F^y) \mb{y}^{\rm c}$, we arrive at
\begin{equation}\label{eq:blockdiagonal_system_2d}
 B \mb{y}^{\rm c} = \mb{c}^{\rm c},
\end{equation}
with $\mb{c}^{\rm c} = (P^x \otimes P^y) \mb{b}^{\rm c}$ and
\[
B = (P^x \otimes P^y)  A^{\rm c} (F^x \otimes F^y) = (P^x A^{{\rm c}x} F^x ) \otimes (P^y A^{{\rm c}y} F^{y} ) =  B^x \otimes B^y.
\]
Note that $F^x \otimes F^y$ is exactly a 2D FFT. The least squares system~\eqref{eq:blockdiagonal_system_2d} is solved by $\mb{y}^{\rm c} = B^\dagger\mb{c}^{\rm c}$. The pseudo-inverse of $B$ also has product structure, namely
\begin{equation}\label{eq:B_pinv_2d}
    B^\dagger = B^{x\dagger} \otimes B^{y\dagger}  = ((B^{x*} B^x)^{-1} B^{x*}) \otimes ((B^{y*} B^y)^{-1} B^{y*}).
\end{equation}

\subsection{Non-periodic approximation problem and AZ algorithm}
The domain $\Omega \subset \Omega_{\rm B}$ may have any shape. Thus, we choose from the grid points $\tilde{X}  \times \tilde{Y}$ the \emph{collocation points} as the subset in $\Omega$,
\begin{equation}\label{eq:set_X_2d}
 X = (\tilde{X}  \times \tilde{Y}) \cap \Omega,
\end{equation}
and this no longer corresponds to a Cartesian grid. We denote the cardinality of the set by $M = \#X$. The oversampling factors $s^x$ and $s^y$ are chosen large enough such that $M > N^x N^y$. We have
\[
M \approx \frac{ L^x L^y S(\Omega)}{S(\Omega_{\rm B})} =  \frac{ s^x s^y N^x N^y S(\Omega)}{S(\Omega_{\rm B})},
\]
so that it is sufficient that $\displaystyle s^x s^y > S(\Omega_{\rm B})/S(\Omega)$.

The discrete least squares approximation problem leads to the linear system
\begin{equation}\label{eq:Aa_is_b_2d}
    A \mb{a} = \mb{b},
\end{equation}
with rectangular matrix $A \in \mathbb{R}^{M \times N^x N^y}$ and with $\mb{b}_i = f(x_i, y_i)$, $i=1, \ldots, M$.
Similar to the univariate case, it is embedded within the block-circulant\footnote{We say block-circulant as a shorthand: more precisely, the matrix represents the tensor product of two block-column circulant matrices.} matrix $A^{\rm c}$. Denote by $R \in \mathbb{R}^{M \times L^x L^y}$ the corresponding restriction of the rows, i.e., such that $A = RA^{\rm c}$, and denote the corresponding extension operator by $E$. 

With this construction, we can still use Algorithm~\ref{alg: az_1d} to solve \eqref{eq:Aa_is_b_2d}. Starting from decomposing matrix $A$
\[
A = R A^{\rm c} = R (P^x \otimes P^y)^* B (F^x \otimes F^y)^*,
\]
and intuitively choosing 
\[
Z^* = (F^x \otimes F^y) B^\dagger (P^x \otimes P^y) E,
\]
the left hand side of the first step in AZ algorithm becomes
\begin{align*}
    A-AZ^*A = R  &(P^x \otimes P^y)^* B  (F^x \otimes F^y)^* - \\ 
    & R  (P^x \otimes P^y)^* B  B^\dagger  (P^x \otimes P^y) E R  (P^x \otimes P^y)^* B  (F^x \otimes F^y)^*,
\end{align*}
with rank much lower than $A$.

\subsection{Efficient AZ algorithm}
\label{ss:alg_eff_2d}
The implementation of the matrix-vector products in the 2D case is of course somewhat more involved than in 1D, but conceptually similar. Note that all operators have product structure, except for the extension and restriction operators $E$ and $R$.

We can further unify the notations with the 1D case using the matrices $P = P^x \otimes P^y$ and $F = F^x \otimes F^y$, and the constants $N = N^x N^y$, $s = s^x s^y$ and $L = L^x L^y = s^x s^y N^x N^y = s N$.

We focus on two aspects that are specific to 2D.

\subsubsection{The choice of $\tau_0$}
As discussed in Lemma~\ref{theo:tau}, the maximal entry of $Z^*$ for 1D problem is $\mathcal{O}(\tau_0^{-1})$ and the maximal entry of $Z^*$ for 2D problem is $\mathcal{O}(\tau_0^{-2})$. However, the influence of $\lVert Z^* \rVert $ remains bounded if  $\eta \sim \tau_0^2$. 

Thus we can choose a larger $\tau_0$ for 2D problem, maybe square root of $\tau_0$ for 1D problem. For example, in the numerical experiments, we choose $\tau_0 =$ 1e-10 for 1D problems and $\tau_0 = $ 1e-5 for 2D problems.

\subsubsection{Rank of the system in step 1}
The rank of the system in step 1 of AZ algorithm was shown to be essentially due to the number of basis functions that overlap with the boundary in the proof of Theorem~\ref{theo:rank}. In 2D, that set is much larger. With $N = N^x N^y$ degrees of freedom, one may generically expect ${\mathcal O}(\sqrt{N})$ basis functions overlapping with the boundary.

That result agrees with the complexity of the rank in the case of 2D Fourier extensions~\cite{matthysen2017fastfe2d} and approximations based on splines~\cite{coppe2022splines}. Since the computational cost of the solver in step 1 scales as ${\mathcal O}(rsN \log N)$ + ${\mathcal O}(M r^2)$ with  $M = {\mathcal O}(N)$ and $r = {\mathcal O}(\sqrt{N})$, the solver has ${\mathcal O}(N^2)$ complexity. While this is better than the cubic cost of a direct solver for the least squares problem, unfortunately, the relative gain of AZ in 2D is smaller than it is in 1D.

\subsection{Numerical results}
Fig.~\ref{fig:d2_fa_per} shows the results of a periodic problem, approximating function $f(x, y) = \sin \left( \left[ \frac{N}{10} \right] \pi (x+y) \right) $ in $\Omega_{\rm B}$ with $T^x = T^y=1$, $s^x = s^y=2$ and $\tau_0 = $1e-5. We can see from the figures that the maximum error of MATLAB's backslash and efficient solver converges to around 1e-12, but the efficient solver is much faster than using backslash. The coefficient norm keeps below 10, normalized by $\sqrt{N}$, which means both solvers are stable.

The left panel in Fig.~\ref{fig:d2_fa_rank} shows the rank of matrices $A$ (blue circle line) and $A-AZ^*A$ (red star line) for non-periodic problems. Since the rank of  $A-AZ^*A$ increases as $\mathcal{O}(\sqrt{N})$, the yellow solid line shows $22 \sqrt{N} - 120 $, which we choose as the column size of randomized matrix. The right panel shows the singular values of matrices $A$ and $A-AZ^*A$ when $N=2500$.

Fig.~\ref{fig:d2_fa_nonper} shows the results of a non-periodic problem, approximating function $f(x,y) = \sin ( \frac{N^x}{10} x+ \frac{N^y}{10} y)$, in the domain $\Omega = \{ (x,y) | x^2 + 4 y^2  \leq 1 \}$,  with $s^x = s^y = 2$, $T^x =1.4$, $  T^y = 0.7$, and $\tau_0 = $1e-5. 
The left panel shows the solution of AZ algorithm on the ellipse. In this panel, the $N^x$ is chosen to be 100 and $N^y$ is chosen to be 50. The $L^\infty$ error of function value is 4.0727e-07. 
The right panel shows the calculating time. 
The blue star line and red solid line  show the results of different solvers for the approximation problem. The maximum error of these two solvers all converges to around 1e-7. AZ algorithm does gain some efficiency compared to using backslash when $N$ is larger than 2500. The coefficient norm keeps around 20, normalized by $\sqrt{N}$, indicating that the solvers are stable.

\begin{figure}[H]
\centering
\includegraphics[width=\textwidth]{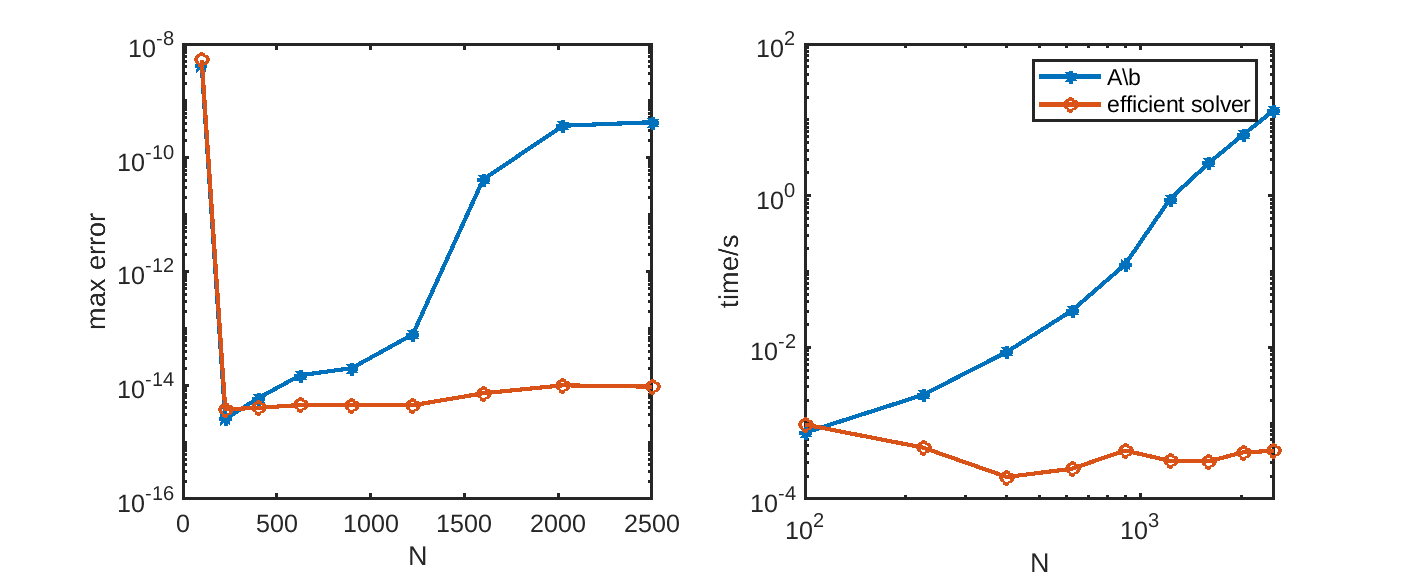}
\caption{Periodic function approximation: $f(x,y) = \sin \left( \left[ \frac{N}{10} \right] \pi (x+y) \right) $ in $\Omega_{\rm B}$ , $T^x = T^y=1$,  $s^x = s^y=2$. Left: max error. Right: calculating time.  Blue star line: backslash in MATLAB. Red circle line: efficient solver.}
\label{fig:d2_fa_per}
\end{figure}

\begin{figure}[H]
\centering
\includegraphics[width=\textwidth]{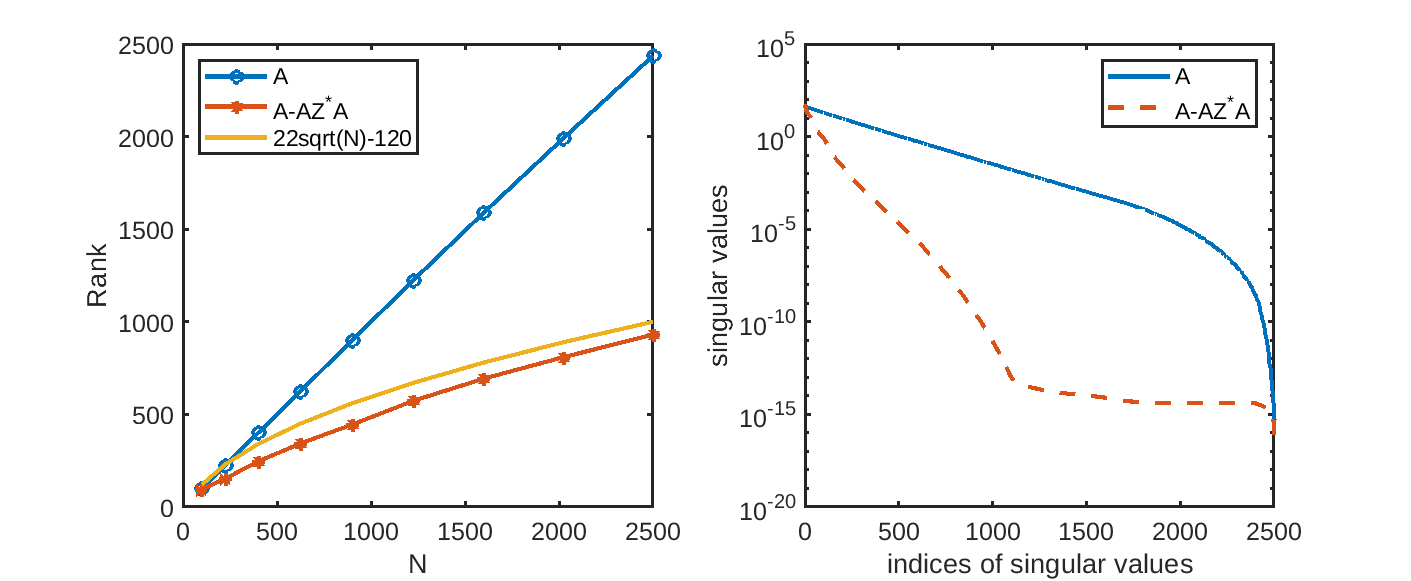}
\caption{ Left: rank of matrices $A$ and $A - A Z^* A$ for 2D non-periodic function approximation problem. Right: singular values of these two matrices for $N^x = N^y = 50$.}
\label{fig:d2_fa_rank}
\end{figure}

\begin{figure}[H]
\centering
\includegraphics[width=\textwidth]{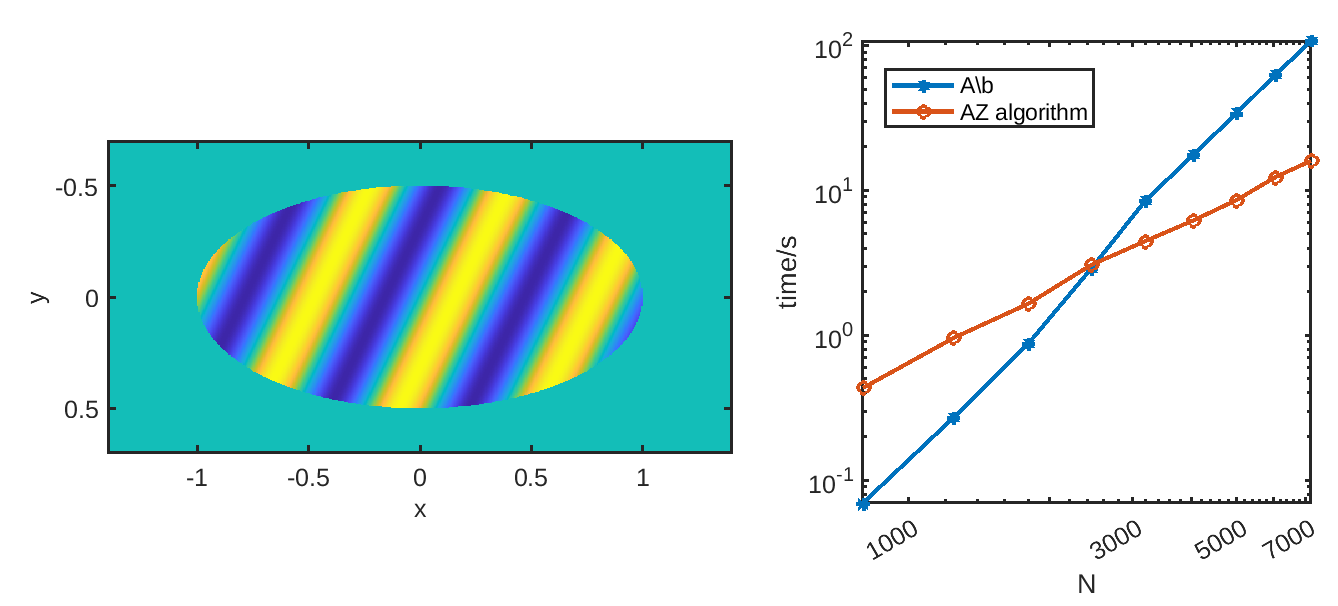}
\caption{Non-periodic function approximation on an ellipse: $f(x,y) = \sin ( \frac{N^x}{10} x+ \frac{N^y}{10} y)$, $T^x =1.4$, $  T^y = 0.7$, $s^x=s^y=2$. Left: values of the function on the ellipse. Right: calculating time. Blue star line: backslash in MATLAB. Red circle line: AZ algorithm.}
\label{fig:d2_fa_nonper}
\end{figure}

\section{Boundary value problems}
\label{ss:bvp}
The solution of boundary value problems is far more involved than the approximation of functions. In this section we explore the possibility of using the same AZ algorithm in unmodified form, by viewing the combination of a differential equation with boundary conditions as a low-rank perturbation of an approximation problem involving translated kernels.

\subsection{An ordinary differential equation}
\subsubsection{Formulation as a least squares problem}
We write an ordinary differential equation (ODE) on $[-1,1]$ as follows,
\begin{equation}\label{eq:ODE}
    \begin{cases}
        \mathcal{L} f(x) = g(x), \quad  x\in \Omega=[-1,1], \\
        \mathcal{B} f(x) = h(x), \quad x \in \partial\Omega = \{-1,1\}.
    \end{cases}
\end{equation}
Here, $\mathcal{L}$ and $\mathcal{B}$ are differential operators, with the latter representing a boundary condition, and the functions $g(x)$ and $h(x)$ are given.
We set out to solve the problem using a least squares collocation approach. This results in a block matrix, in which a block corresponds to (weakly) enforcing the differential equation at the collocation points, and another block corresponds to (weakly) enforcing the boundary conditions.

We select $M$ collocation points $X$ and $N$ basis functions $\{ \phi_k \}_{k=1}^N$ exactly as in the approximation problems of this paper, and denote the endpoints by the set $X^{\rm b} := \{  -1,1 \}$. The discrete least-squares approximation problem for~\eqref{eq:ODE} results in the rectangular system
\[
 \begin{bmatrix} A^{\rm i} \\ A^{\rm b} \end{bmatrix} \mathbf{a} = \begin{bmatrix} \mb{b}^{\rm i} \\ \mb{b}^{\rm b} \end{bmatrix},
\]
with $A^{\rm i}_{i j} = \mathcal{L} \phi_j(  x_i)$ and $A^{\rm b}_{i j} = \mathcal{B} \phi_j(x^{\rm b}_i)$. The right hand side corresponds to the inhomogeneous term and the boundary condition respectively, i.e., $\mb{b}^{\rm i}_i = g( x_i)$ and $\mb{b}^{\rm b}_i = h( x^{\rm b}_i)$.
This is an approximation problem with basis functions $\{ \mathcal{L} \phi_j \}$.

\begin{example}
\label{eg:ode}
As a simple illustration, we consider the one-dimensional Helmholtz equation $u'' + k^2 u = 0$ with wavenumber $k$ and boundary conditions $u(\pm 1) = \sin(\pm N/5)$. The analytical solution is $u(x) = \sin (N x /5)$. In this case, $A^{\rm i}$, the upper $M \times N$ subblock of the least squares matrix $A$ has entries
\[
A_{i j} = \phi_j '' (x_i) + k^2 \phi_j(x_i), \quad i = 1,2,\cdots,M, \quad j=1,2,\cdots,N.
\]
The Gaussian RBF~\eqref{eq:rbf} has second derivative
\begin{equation*}
    \phi''(r) = -2 \varepsilon^2 \exp ( -\varepsilon^2 r^2) (1 - 2\varepsilon^2 r^2).
\end{equation*}
In order to avoid large matrix entries in $A$, we normalize $A^{\rm i}$ with the factor $-2 \varepsilon^2$, and do the same operation with $\mb{b}^{\rm i}$.
\end{example}

Computational results for this example problem are given in the next section.

\subsubsection{AZ algorithm}

The first block $A^{\rm i}$ of the block matrix $A$ has the same structure as an approximation problem with translated kernels, the kernels being given by $\mathcal{L} \phi$. The boundary condition results in a small number of additional rows, which we treat as a low-rank perturbation of the approximation problem. Hence, with minor modification to $Z^*$ to account for the different dimension, we can apply the AZ algorithm.

Since $A = \begin{bmatrix} A^{\rm i} \\ A^{\rm b} \end{bmatrix} = \begin{bmatrix} R P^* B F^* \\ A^{\rm b} \end{bmatrix} $, we choose the matrix $Z^*$ with two additional zero columns (the same as the number of boundary points):
\[
Z^* = 
\begin{bmatrix}
    F B^\dagger P E & \mb{0}_{N \times 2}
\end{bmatrix}.
\]
Here $F$, $P$, $B^\dagger$ and $E$  have the same definition as in Sec.~\ref{ss:eff_alg_1d}.

Somewhat surprisingly, this is the only modification to the AZ algorithm necessary to solve ODE's such as the one of Example~\ref{eg:ode}.

\subsection{An elliptic boundary value problem in 2D}
\subsubsection{Least-squares approximation}
\label{ss:pbl-pde-2d}
The approach for a bivariate boundary value problem, involving a partial differential equation (PDE) rather than an ODE, is entirely similar but requires more boundary conditions.

Consider a 2D boundary value problem involving the Helmholtz equation:
\[
    \Delta u(x,y) + k_0^2 u(x,y) =    \frac{\partial^2 u}{\partial  x^2} +   \frac{\partial^2 u}{\partial  y^2} + k_0^2 u(x,y) = 0, \qquad (x,y) \in \Omega,
\]
in which the domain $\Omega \subseteq \Omega_{\rm B} = [-T^x, T^x] \times [-T^y, T^y]$ is embedded in a rectangle, and with Dirichlet boundary conditions $u(x,y) = g(x,y)$ on $(x,y) \in \Gamma = \partial \Omega$. We can also choose other types of boundary conditions without changes to the methodology. Here we choose a Dirichlet boundary condition for simplicity.

Similar to the 2D approximation problem we choose $M$ collocation points in $\Omega$ as a subset of an equispaced grid on the covering rectangle. In addition, we choose $M^{\rm b}$ points along the boundary. A discrete least squares systems follows by enforcing the Helmholtz equation to hold on the grid, and the boundary condition to hold at the boundary points. This leads to $A \mb{a} = \mb{b}$ with $A = \begin{bmatrix}
A^{\rm i} \\ A^{\rm b} \end{bmatrix} $ consisting of two subblocks, $A^{\rm i} \in \mathbb{R}^{M \times N^x N^y}$ for collocation points and $A^{\rm b} \in \mathbb{R}^{M^{\rm b} \times N^x N^y}$ for boundary points. This system is structurally similar to that of the ODE, but the block $A^{\rm b}$ has larger dimension since there are more boundary points.

The entries of the blocks are given in this case by
\begin{gather}
A^{\rm i}_{i,:} = 
    \begin{bmatrix}  {\phi_1^x}'' (x_i) & {\phi_2 ^x}'' (x_i) & \cdots & {\phi_{N^x}^x}'' (x_i)    \end{bmatrix}
     \otimes
    \begin{bmatrix}  \phi_1^y (y_i) & \phi_2 ^y (y_i) & \cdots & \phi_{N^y}^y (y_i)    \end{bmatrix} \label{eq:Ai_2d_bvp}\\
     +
    \begin{bmatrix}  \phi_1^x (x_i) & \phi_2 ^x (x_i) & \cdots & \phi_{N^x}^x (x_i)    \end{bmatrix}
    \otimes
    \begin{bmatrix}  {\phi_1^y}'' (y_i) & {\phi_2 ^y}'' (y_i) & \cdots {\phi_{N^y}^y}''  (y_i)    \end{bmatrix} \nonumber\\
     + k_0^2
    \begin{bmatrix}  \phi_1^x (x_i) & \phi_2 ^x (x_i) & \cdots & \phi_{N^x}^x (x_i)    \end{bmatrix}
    \otimes
    \begin{bmatrix}  \phi_1^y (y_i) & \phi_2 ^y (y_i) & \cdots & \phi_{N^y}^y (y_i)    \end{bmatrix}, \nonumber \\
    i = 1, \cdots, M. \nonumber
\end{gather}
and
\begin{gather*}
A^{\rm b}_{i,:} = 
    \begin{bmatrix}  \phi_1^x (x_i) & \phi_2 ^x (x_i) & \cdots & \phi_{N^x}^x (x_i)    \end{bmatrix}
    \otimes
    \begin{bmatrix}  \phi_1^y (y_i) & \phi_2 ^y (y_i) & \cdots & \phi_{N^y}^y (y_i)    \end{bmatrix}, \\
    i = 1, \cdots, M^{\rm b}.
\end{gather*}

The right hand side vector similarly has block structure: $\mb{b} = \begin{bmatrix} \mb{b}^{\rm i} \\ \mb{b}^{\rm b} \end{bmatrix} $ has two sub-vectors, namely $\mb{b}^{\rm i} = \mb{0} \in \mathbb{R}^{M}$ for collocation points and $\mb{b}^{\rm b} \in \mathbb{R}^{M^{\rm b}}$ for boundary points, with $\mb{b}^{\rm b}_i = g(x_i,y_i)$.

\subsubsection{AZ algorithm}
The AZ algorithm is the same as it is for the ODE case, but it requires more motivation. 

We first look at the structure of matrix $A^{\rm i}$, which can be written as
\[
A^{\rm i} = R A^{\rm ic},
\]
where $R$ is the matrix to select rows according to the collocation points, and $A^{\rm ic} \in \mathbb{R}^{s^x s^y N^x N^y \times N^x N^y}$ consists of evaluations at the full set of equispaced collocation points in the larger box $\Omega_{\rm B}$. The entries of $A^{\rm ic}$ are given in this case by
\begin{gather*}
A^{\rm ic}_{i,:} = 
    \begin{bmatrix}  {\phi_1^x}'' (x_i) & {\phi_2 ^x}'' (x_i) & \cdots & {\phi_{N^x}^x}'' (x_i)    \end{bmatrix}
     \otimes
    \begin{bmatrix}  \phi_1^y (y_i) & \phi_2 ^y (y_i) & \cdots & \phi_{N^y}^y (y_i)    \end{bmatrix}   \\
     +
    \begin{bmatrix}  \phi_1^x (x_i) & \phi_2 ^x (x_i) & \cdots & \phi_{N^x}^x (x_i)    \end{bmatrix}
    \otimes
    \begin{bmatrix}  {\phi_1^y}'' (y_i) & {\phi_2 ^y}'' (y_i) & \cdots {\phi_{N^y}^y}''  (y_i)    \end{bmatrix} \\
     + k_0^2
    \begin{bmatrix}  \phi_1^x (x_i) & \phi_2 ^x (x_i) & \cdots & \phi_{N^x}^x (x_i)    \end{bmatrix}
    \otimes
    \begin{bmatrix}  \phi_1^y (y_i) & \phi_2 ^y (y_i) & \cdots & \phi_{N^y}^y (y_i)    \end{bmatrix},   \\
    i = 1, \cdots, s^x s^y N^x N^y. 
\end{gather*}
Thus the matrix can be written as the summation of three tensor products
\[
A^{\rm ic} = A^{x {\rm ic}}_2 \otimes A^{y {\rm ic}}_0 + A^{x {\rm ic}}_0 \otimes A^{y {\rm ic}}_2 + k_0^2 A^{x {\rm ic}}_0 \otimes A^{y {\rm ic}}_0,
\]
where the superscripts $x$ and $y$ indicate the basis function and the subscripts show their order of derivative. This structure is more complicated than that of the 2D approximation problem. However, all three matrices can be simultaneously brought into block-column diagonal form:
\begin{align*}
B &= (P^{x} \otimes P^{y}) A^{\rm ic} (F^x \otimes F^y) \\
&= (P^x A^{x {\rm ic}}_2 F^x) \otimes (P^y A^{y {\rm ic}}_0 F^y) + 
(P^x A^{x {\rm ic}}_0 F^x) \otimes (P^y A^{y {\rm ic}}_2 F^y) \\
& \qquad + k_0^2 (P^x A^{x {\rm ic}}_0 F^x) \otimes (P^y A^{y {\rm ic}}_0 F^y) \\
&= B_2^x \otimes B_0^y + B_0^x \otimes B_2^y + k_0^2 B_0^x \otimes B_0^y.
\end{align*}
Although the matrix $B$ itself can not be written in the form of a tensor product, it is highly sparse and its structure enables efficient computations. The pseudo-inverse is
\[
B^\dagger = (B^* B)^{-1} B^*.
\]
Here, $B^* B$ is a diagonal matrix, of which the diagonal can be computed from those of the blocks in the three components above by straightforward calculation. 

The formulation for solving the PDEs using the AZ algorithm is entirely similar to the one for solving ODEs: we simply add zero columns to $Z^*$ to account for the increased dimension of $A^{\rm b}$. Since $A^{\rm i}$ can be decomposed into $A^{\rm i}= R P^* B F^*$, we define matrix $Z^*$ as 
\[
Z^* = 
\begin{bmatrix}
    F B^\dagger P E     &  \mb{0}_{N \times M^{\rm b} }
\end{bmatrix}.
\]
We make two additional comments.
\begin{itemize}
\item First, we observe that the boundary conditions result in a larger block of the system, compared to the two rows of the ODE. Since we treat that block as a low-rank perturbation, it will increase the rank of the first step of the AZ algorithm. Thus, the AZ algorithm will be less efficient. However, that was already the case with 2D approximations versus 1D approximations and the difference is only a constant factor, since one can expect the size of the boundary block to be comparable to the rank of the plunge region of the 2D approximation problem (which itself scales with the length of the boundary).

\item Second, it is not trivial to see why the AZ algorithm works at all with the same choice of $Z$ as for approximation problems. Indeed, the expression~\eqref{eq:Ai_2d_bvp} of the matrix entries shows that we no longer have an approximation problem involving shifts of a single kernel function. Instead, we see combinations of kernels involving the derivatives of $\phi$ in the $x$ and $y$ directions. However, as our derivation above shows, the matrix $A^{\rm i}$ is the sum of matrices that do have the structure of a 2D approximation problem. Thus, each of the three terms in~\eqref{eq:Ai_2d_bvp} results in a matrix that can be embedded in a larger tensor product of block-column circulant matrices. This implies that their sum, the matrix $A^{\rm i}$ itself, is still embedded in a larger matrix with special structure. Hence, we can define $Z^*$ in exactly the same way as before.
\end{itemize}

\subsection{Numerical results}
Fig.~\ref{fig:d1_ode_nonper} shows the results of the proposed modification of the AZ algorithm to Example~\ref{eg:ode}, with $T=1.5$ and $s=2$. The maximum errors in both cases is below $1e-8$. For reasons discussed in Sec.~\ref{ss:alg_eff_1d}, the computational complexity of the AZ algorithm is $\mathcal{O}(N \log N)$, which is near optimal and a significant improvement compared to a direct dense solver. Finally, the coefficient norms in both cases remain moderate, which indicates that the solver is stable and accurate.

Fig.~\ref{fig:d2_pde_nonper} shows the results of solving a two-dimensional boundary value problem involving the Helmholtz equation as described in Sec.~\ref{ss:pbl-pde-2d}, with the circular domain $\Omega = \{ (x,y) | x^2 + y^2  \leq 1 \}$, $T^x = T^y = 1.5$, $s^x = s^y = 2$ and using $100$ equispaced boundary points on $\partial\Omega$. Here, too, the coefficient norms remain moderate, indicating that the solvers are stable (not shown).

\begin{figure}[H]
\centering
\includegraphics[width=\textwidth]{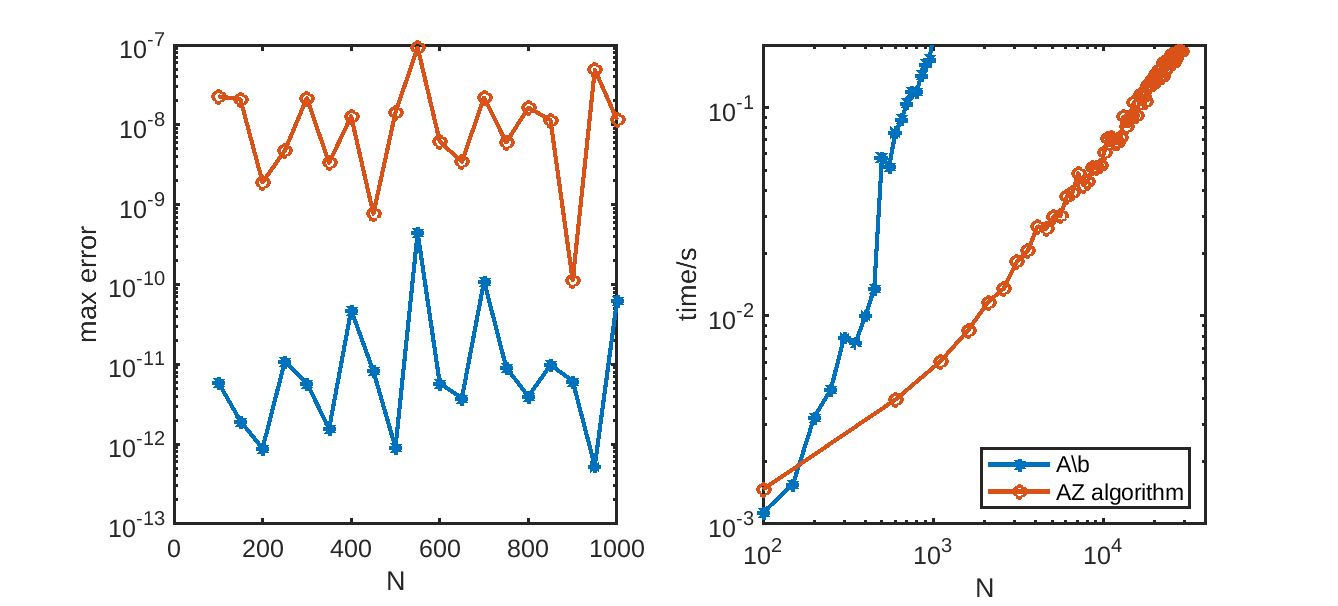}
\caption{Solving ODE $u''+k^2 u =0$ on $[-1,1]$, with solution $u(x) = \sin \left(  \frac{N}{5} x \right) $, $T=1.5$, $s=2$, $\tau_0 = 10^{-10}$. Left: maximum error on $[-1,1]$ as a function of $N$. Right: computing time.  Blue star line: backslash in MATLAB. Red circle line: AZ algorithm.}\label{fig:d1_ode_nonper}
\end{figure}

\begin{figure}[H]
\centering
\includegraphics[width=\textwidth]{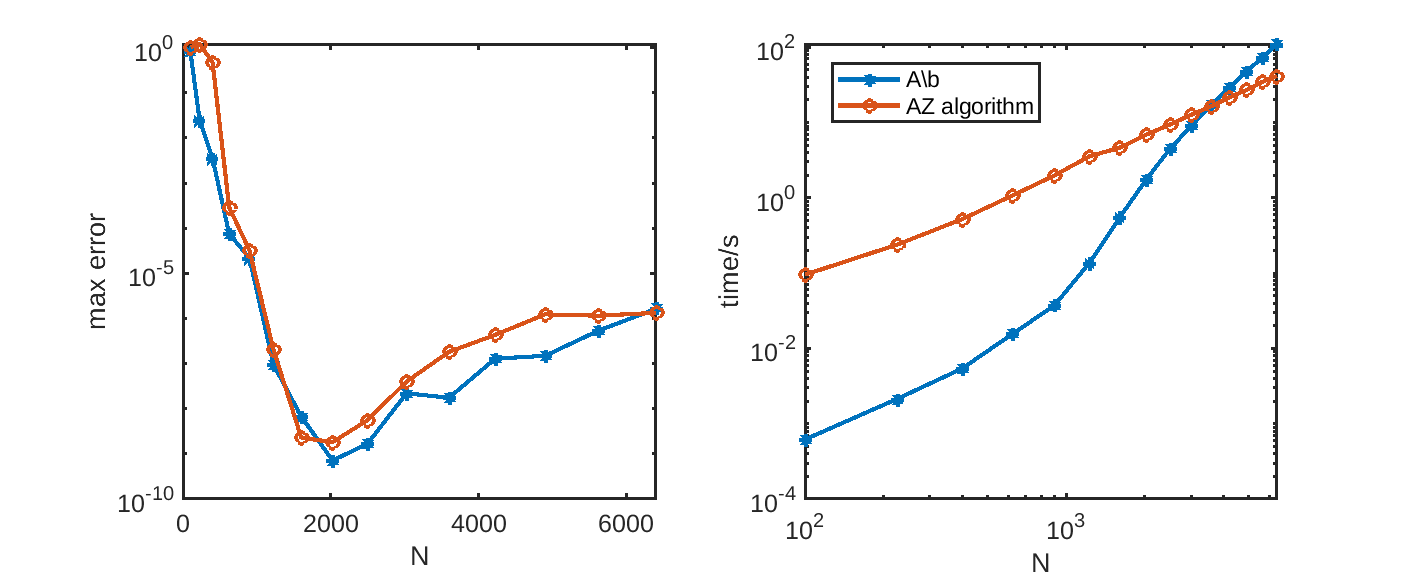}
\caption{Solving PDE: $\nabla^2 u(x,y) + k_0^2 u(x,y) =0$ in the unit circle, with solution $u(x) = \sin (2x+3y) $, $k^2_0 = 13$, $T^x = T^y=1.5$, $s^x = s^y=2$, $\tau_0 = 10^{-5}$. Left: max error. Right: calculating time.  Blue star line: backslash in MATLAB. Red circle line: AZ algorithm.}\label{fig:d2_pde_nonper}
\end{figure}

As another example, we solve the Helmholtz equation with different boundary conditions
\begin{gather*}
    \nabla^2 u(x,y)  + 4 u(x,y) = \exp{ \left( -4((x+0.3)^2 + y^2)^2 \right)}, \quad  (x,y) \in \Omega, \\
    \frac{\partial u(x,y)}{\partial \mathbf{n}} = 0,  \quad  (x,y) \in \partial\Omega_1, \\
    u(x,y) = 0,  \quad  (x,y) \in \partial\Omega_2.
\end{gather*}
We choose $\Omega$ to be a two-dimensional smooth flower-shaped domain with a circle of diameter $0.2$ cut out, with homogeneous Neumann boundary conditions on the outer boundary $ \partial\Omega_1$ and homogeneous Dirichlet boundary conditions on the inner boundary $ \partial\Omega_2$.

The solution calculated by AZ algorithm is shown in Fig.~\ref{fig:d2_pde_flower}. The approximation uses $100 \times 100$ basis functions on $[-1,1]\times[-1,1]$ grid (slightly larger than the flower) with oversampling factor 2 in both dimensions. For this experiment we have selected $200$ and $100$ points on the outer and inner boundary respectively. The computation time for this experiment using AZ was $87\,{\rm s}$, which is approximately three times faster than $245\,{\rm s}$ when using backslash in MATLAB for the least squares problem. The dimension of $A$ in this case is $15071 \times 10000$.

\begin{figure}[H]
    \centering
    \includegraphics[width=\textwidth]{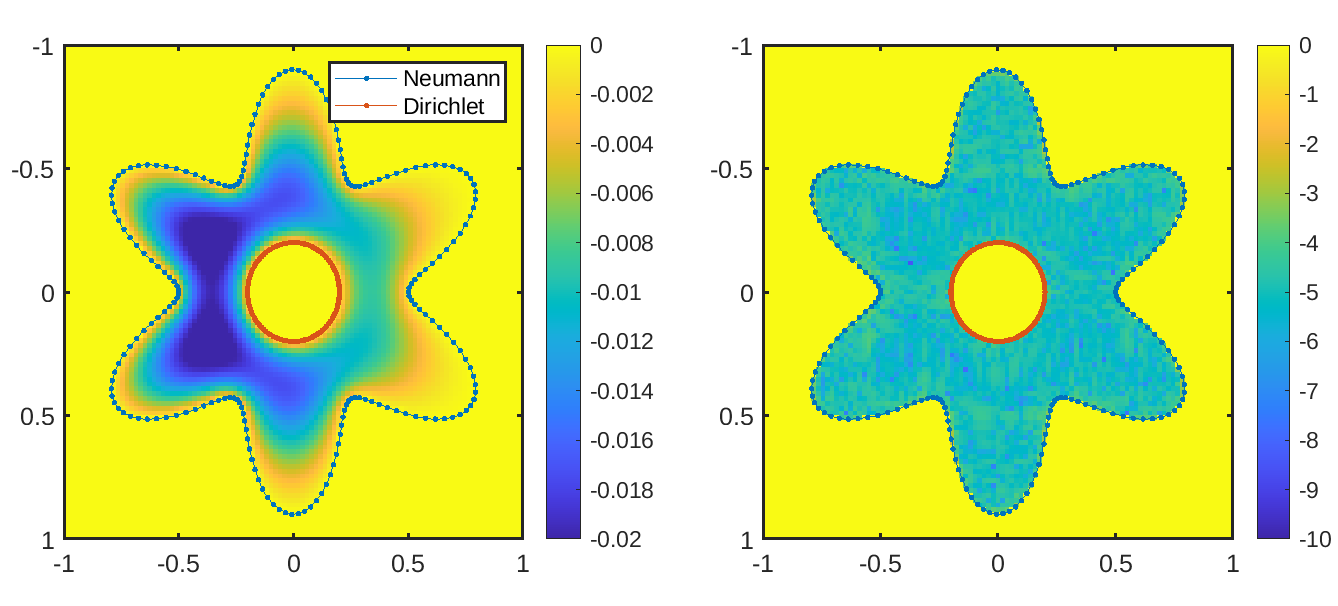}
    \caption{Left: solution of the Helmholtz equation on a flower-shaped domain. Right: absolute error in logarithmic scale at 100 $\times$ 100 equispaced points on $[-1,1]\times[-1,1]$.}
    \label{fig:d2_pde_flower}
\end{figure}

It is known in the solution of approximation problems involving ill-conditioned systems that the accuracy of the solution is affected by the size of the coefficient norm (recall \S\ref{ss:theo} and see~\cite{adcock2022rbf_frames}). We note that in this final example the solution computed by Matlab has a coefficient norm of $5.5e4$, whereas that of the AZ computed solution was around $59$.

\section{Concluding remarks}
\label{ss:con}
The theme of this article has been to use periodized Gaussian radial basis functions for function approximation and to devise an efficient algorithm to do so. The method is extended to solve boundary value problems on irregular (i.e., non-cartesian) domains.

The efficient solver is based on the properties of block-column circulant matrices and low-rank perturbations thereof. However, a number of open problems remain for future research:
\begin{itemize}
    \item For two-dimensional problems, the rank of the perturbation scales as $\mathcal{O}(\sqrt{N})$, and this leads to $\mathcal{O}(N^2)$ overall complexity. While this improves upon a direct solver and other spectral methods, it is a relevant problem to investigate further complexity improvements. To that end, one could investigate the structure of the approximation problem in step 1 of the AZ algorithm, which is itself a univariate problem along the boundary of the domain.
    \item The computational results in 2D have been limited to the Gaussian RBF, since among all popular RBFs only the Gaussian RBF can be written as the product of univariate functions. Can the method be extended to other RBFs?
    \item The efficient solver for boundary value problems in this article only applies to problems involving a constant coefficient differential operator. Can the AZ algorithm be used for boundary value problems with non-constant coefficients? In this case one may need a different methodology to find a suitable $Z$ matrix.
\end{itemize}

\bibliographystyle{abbrv}
\bibliography{references}

\end{document}